\documentclass[11pt]{article}
\usepackage[utf8]{inputenc}
\usepackage{amsmath}
\usepackage{amsfonts}
\usepackage{amsthm,subcaption}
\usepackage{amssymb}
\usepackage{enumitem}
\usepackage[margin=1in]{geometry} 
\usepackage{color}
\usepackage{setdeck}
\usepackage{graphicx}
\usepackage{comment}
\usepackage{setspace}

\usepackage{tikz-qtree}
\usetikzlibrary{shadows,trees}

\newcommand{\ignore}[1]{}

\renewcommand{\P}{\mathcal{P}}

\newcommand{\R}{\mathbf{R}}
\newcommand{\C}{\mathbf{C}}

\newcommand{\pts}{\mathcal P}
\newcommand{\vb}{{\bf v}}
\newcommand{\Vb}{{\mathbf{V}}}

\newcommand{\parag}[1]{\vspace{2mm}\noindent{\bf #1} }

\newtheorem{theorem}{Theorem}[section]

\newtheorem{corollary}[theorem]{Corollary}
\newtheorem{lemma}[theorem]{Lemma}

\newtheorem*{th:FewDistances}{Theorem \ref{th:FewDistancesBetweenConics}}
\newtheorem*{th:PerpendicularPlanesSpecialForm}{Theorem \ref{th:PerpendicularPlanesSpecialForm}}
\newtheorem*{co:PerpendicularPlanesFewDistances}{Corollary \ref{co:PerpendicularPlanesFewDistances}}
\newtheorem*{th:FewOrMany}{Theorem \ref{th:FewOrMany}}

\def\eps{{\varepsilon}}
\def\bv{\mathbf{v}}

\newcommand{\coeff}{\textit{coeff}\hspace{2pt}}

\title{Distinct Distances in $\R^3$\\Between Quadratic and Orthogonal Curves\thanks{This work was done as part of the 2021 Polymath Jr program, partially supported by NSF award DMS–2113535.} }
\author{Toby Aldape\thanks{Department of Mathematics, University of Texas at Austin, TX, USA. {\sl tmaldape@utexas.edu}.},
\ \ Jingyi Liu\thanks{Department of Computer Science, Princeton University, NJ, USA.
{\sl jl1606@princeton.edu}.},
\ \ Gregory Pylypovych\thanks{Department of Mathematics, Massachusetts Institute of Technology, MA, USA. {\sl gpylypov@mit.edu}.},
\ \ Adam Sheffer\thanks{Department of Mathematics, Baruch College, City University of New York, NY, USA.
{\sl adamsh@gmail.com}. Supported by NSF award DMS-1802059.},
\ \ Minh-Quan Vo\thanks{Department of Mathematics and Computer Science, University of Science, Vietnam National University, Ho Chi Minh City, Vietnam.
{\sl mqmath0000@gmail.com}.}}
\begin{document}

\tikzset{font=\small,
edge from parent fork down,
level distance=1.75cm,
every node/.style=
    {rectangle,
    minimum height=10mm,
    minimum width=10mm,
    draw=blue!80,
    very thick,
    align=center,
    text depth = 2pt
    },
edge from parent/.style=
    {draw=blue!60,
    thick
    }  
}

\date{}

\maketitle
\begin{abstract}
We study the minimum number of distinct distances between point sets on two curves in $\R^3$.
Assume that one curve contains $m$ points and the other $n$ points. Our main results: 

(a) When the curves are conic sections, we characterize all cases where the number of distances is $O(m+n)$. This includes new constructions for points on two parabolas,  two ellipses, and one ellipse and one hyperbola.
In all other cases, the number of distances is $\Omega(\min\{m^{2/3}n^{2/3},m^2,n^2\})$. 

(b) When the curves are not necessarily algebraic but smooth and contained in perpendicular planes, we characterize all cases where the number of distances is $O(m+n)$. This includes a surprising new construction of non-algebraic curves that involve logarithms. 
In all other cases, the number of distances is $\Omega(\min\{m^{2/3}n^{2/3},m^2,n^2\})$. 
\end{abstract}

\section{Introduction}

Erd\H{o}s \cite{erd46} introduced his distinct distances problem in 1946, which has motivated significant mathematical developments in the fields of discrete geometry and additive combinatorics. For a given $n$, the problem asks for the minimum number of distinct distances spanned by $n$ points in $\R^2$. Given a finite point set $\P$, we define $D(\P)$ as the number of distinct distances spanned by pairs of $\P^2$. The distinct distances problem asks for $\min D(\P)$, where the minimum is over all sets $\P\subset\R^2$ of $n$ points. Erd\H{o}s discovered such a set $\P$ that satisfies $D(\P)=O(n/\sqrt{\log n}).$ 
The problem was almost completely resolved by Guth and Katz \cite{GK15}, who proved that every set $\pts$ of $n$ points satisfies $D(\P)=\Omega(n/\log n)$. 

The above is one of many distinct distances problems, most of which were also introduced by Erd\H os (for example, see Sheffer \cite{Sheffer14}). 
In a \emph{bipartite} distinct distances problem, we have two point sets $\pts_1$ and $\pts_2$ and are interested in the distinct distances spanned by the pairs of $\pts_1\times \pts_2$. 
We denote this number as $D(\pts_1,\pts_2)$. 
In one family of bipartite distinct distances problems, we assume that there exist  curves\footnote{Unless stated otherwise, by \emph{curves} we refer to constant-complexity irreducible one-dimensional varieties. See Section \ref{sec:prelim} for the technical definitions of these concepts.} $C_1,C_2$, such that $\pts_1\subset C_1$ and $\pts_2\subset C_2$. 
Purdy conjectured that, when $|\pts_1|=|\pts_2|=n$ and $C_1,C_2$ are lines that are neither parallel nor orthogonal,
$D(\pts_1,\pts_2)$ is superlinear in $n$ (for example, see \cite[Section 5.5]{BMP05}). 
Pach and de Zeeuw \cite{PdZ13} proved the following more general result. 

\begin{theorem}\label{th:PachDeZeeuw}
Let $C_1$ and $C_2$ be curves in $\R^2$ that are not parallel lines, orthogonal lines, or concentric circles. 
Let $\pts_1$ be a set of $m$ points in $C_1$ and let $\pts_2$ be a set of $n$ points in $C_2$.
Then 
\[ D(\pts_1,\pts_2) = \Omega(  \min\{m^{2/3}n^{2/3},m^2,n^2\}). \] 
\end{theorem}
When $C_1$ and $C_2$ are parallel lines, orthogonal lines, or concentric circles, we can choose $\pts_1$ and $\pts_2$ so that $D(\pts_1,\pts_2)=\Theta(m+n)$. It is conjectured that the bound of Theorem \ref{th:PachDeZeeuw} is not tight. 
Very recently, Solymosi and Zahl \cite{SolyZhal22} improved the bound for the case of two lines to $\Omega(\min\{m^{3/4}n^{3/4},m^2,n^2\})$.

We refer to a pair of curves $(C_1,C_2)$ as a \emph{configuration}.
A configuration \emph{spans many distances} if every 
finite $\pts_1\subset C_1$ and $\pts_2\subset C_2$ satisfy $$D(\pts_1,\pts_2)=\Omega(\min\{ |\pts_1|^{2/3}|\pts_2|^{2/3}, |\pts_1|^2, |\pts_2|^2 \}).$$
Theorem \ref{th:PachDeZeeuw} states that every configuration in $\R^2$ spans many distances, unless it consists of parallel lines, orthogonal lines, or concentric circles.  
We say that a configuration \emph{spans few distances} if for every $m$ and $n$ there exist $\P_1\subset C_1$ and $\P_2\subset C_2$ such that $|\pts_1|=m$, $|\pts_2|=n$, and  $D(\P_1,\P_2)=O(m+n)$. 
By the above, every two curves in $\R^2$ span either few or many distances. 
In other words, for any two curves that do not span many distances, there exist point sets that span a linear number of distinct distances.

 \begin{figure}[ht]
    \centering
    \begin{subfigure}[b]{0.25\textwidth}
    \centering
        \includegraphics[width=0.97\textwidth]{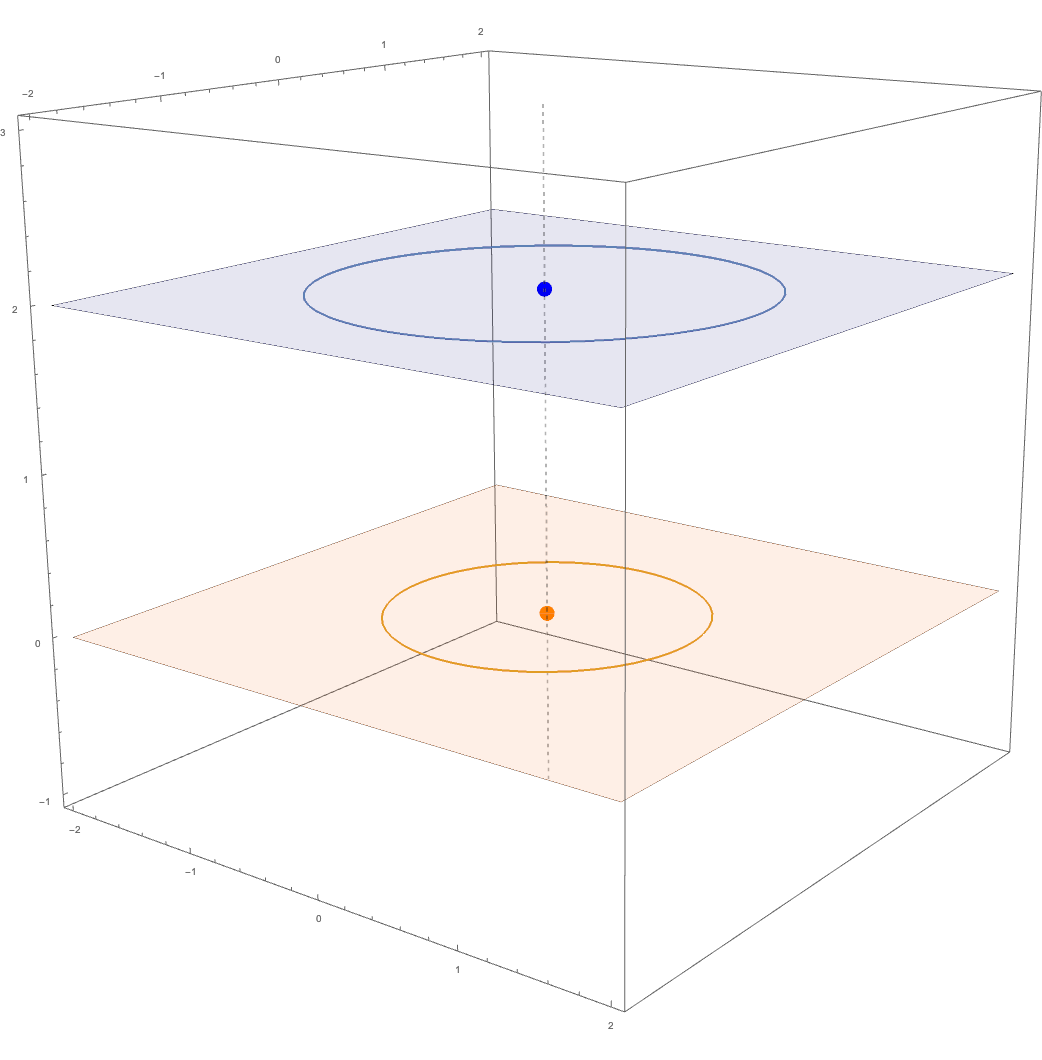}
        \caption{Aligned circles.}
    \end{subfigure}
    \hspace{1cm}
    \begin{subfigure}[b]{0.25\textwidth}
        \centering
        \includegraphics[width=\textwidth]{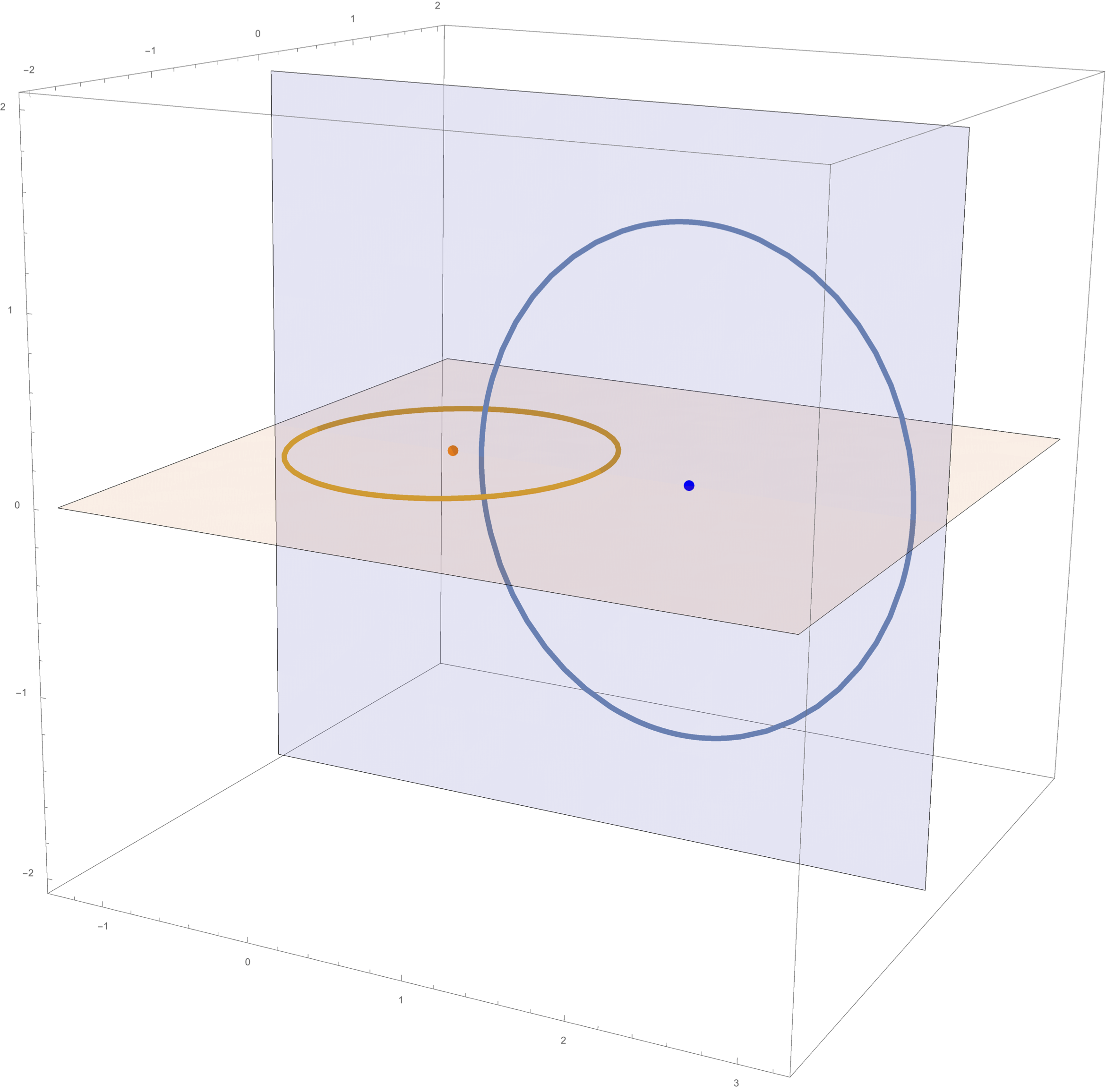}
        \caption{Perpendicular circles.}
    \end{subfigure}
    \caption{Pairs of circles that span few distances.}
    \label{fig:TwoCircles}
\end{figure}

Mathialagan and Sheffer \cite{MS20} studied the case where $C_1$ and $C_2$ are circles in $\R^3$. The \emph{axis} of a circle $C\subset \R^3$ is the line passing through the center of $C$ and orthogonal to the plane that contains $C$. Two circles are $\emph{aligned}$ if they share a common axis. See Figure \ref{fig:TwoCircles}(a). In some sense, the case of aligned circles in $\R^3$ generalizes the case of concentric circles in $\R^2$. In particular, by adapting the construction for concentric circles in $\R^2$, we get that aligned circles span few distances.
 
 Surprisingly, there exists another configuration of circles in $\R^3$ that spans few distances. 
 Two planes in $\R^3$ are \emph{perpendicular} if the angle between them is $\pi/2$.
 Let $H_1$ and $H_2$ be the planes that contain $C_1$ and $C_2$, respectively. We say that circles $C_1$ and $C_2$ are \emph{perpendicular} if  
\begin{itemize}[noitemsep,topsep=1pt]
    \item $H_1$ contains the center of $C_2.$
    \item $H_2$ contains the center of $C_1.$
    \item The planes $H_1$ and $H_2$ are perpendicular. 
\end{itemize}
See Figure \ref{fig:TwoCircles}(b).
Mathialagan and Sheffer \cite{MS20} proved the following theorem.
\begin{theorem}\label{th:ShefferMathialagan}
Let $C_1$ and $C_2$ be two circles in $\R^3$. \\
(a) Assume that $C_1$ and $C_2$ are aligned or perpendicular. Then there exists a set $\P_1\subset C_1$ of $m$ points and $\P_2\subset C_2$ of $n$ points, such that \[ D(\P_1,\P_2)=\Theta(m+n). \] 
(b) Assume that $C_1$ and $C_2$ are neither aligned nor perpendicular. Let $\P_1\subset C_1$ be a set of $m$ points and let $\P_2\subset C_2$ be a set of $n$ points. Then
\[
D(\P_1,\P_2)=\Omega(\min\{m^{2/3}n^{2/3},m^2,n^2\}).
\]
\end{theorem}

Theorem \ref{th:ShefferMathialagan} states that two circles in $\R^3$ span many distances, unless they are aligned or perpendicular, in which case then they span few distances. 

\parag{Our first result: conics.} Continuing where Mathialagan and Sheffer stopped, we establish two new results for distinct distances between curves in $\R^3$. Our first result generalizes Theorem \ref{th:ShefferMathialagan} to  all non-degenerate conic sections. That is, each curve may now be a circle, a parabola, a hyperbola, or an ellipse.
To state this result, we first need to define a few new configurations.

The \emph{axis} of a parabola is the line incident to both the vertex and focus of that parabola. For example, see Figure \ref{fig:TwoParabolas}(a).
We say that parabolas $C_1,C_2$ in $\R^3$ are \emph{congruent and perpendicularly opposite} if they satisfy:
\begin{itemize}[noitemsep,topsep=1pt]
    \item We can obtain $C_2$ by translating and rotating $C_1$.
    \item The parabolas $C_1$ and $C_2$ have the same axis.
    \item The planes that contain $C_1$ and $C_2$ are perpendicular. 
    \item For each parabola, consider the vector from its vertex to its focus. These two vectors have opposite directions. 
\end{itemize}
For example, see Figure \ref{fig:TwoParabolas}(b).
For short, we refer to such a pair of parabolas as \emph{CPO parabolas}.

 \begin{figure}[ht]
    \centering
    \begin{subfigure}[b]{0.2\textwidth}
    \centering
        \includegraphics[width=0.97\textwidth]{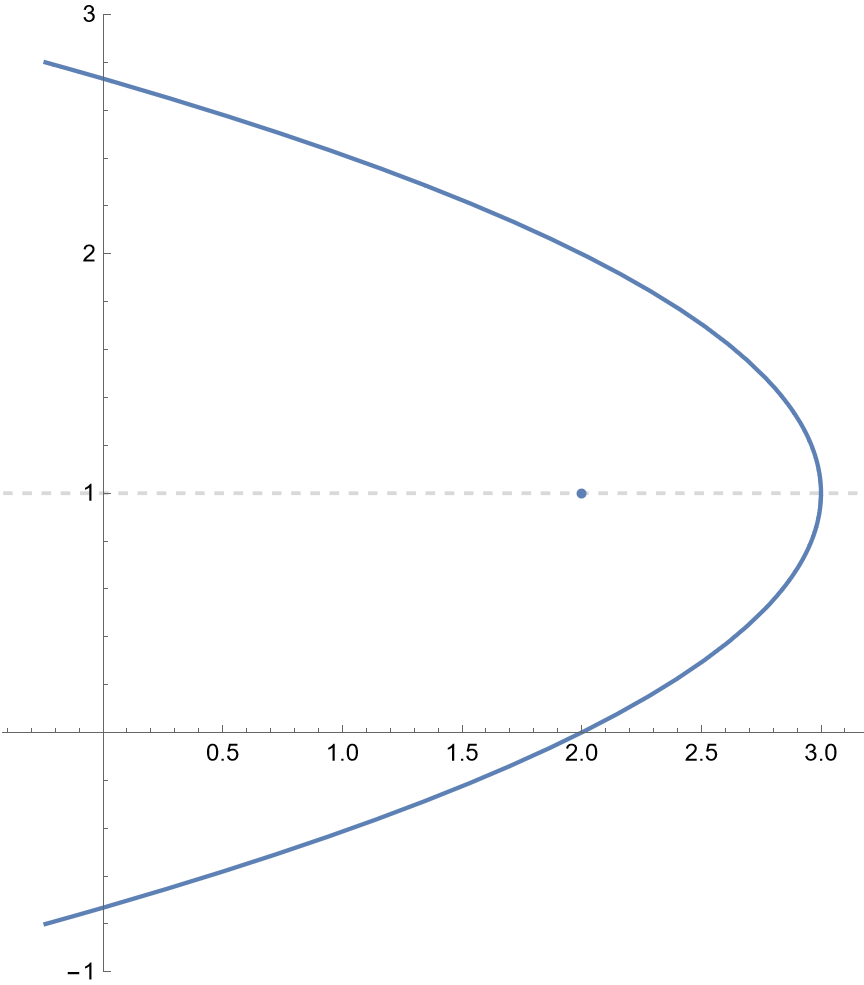}
        \caption{The axis of a parabola.}
    \end{subfigure}
    \hspace{1cm}
    \begin{subfigure}[b]{0.3\textwidth}
        \centering
        \includegraphics[width=\textwidth]{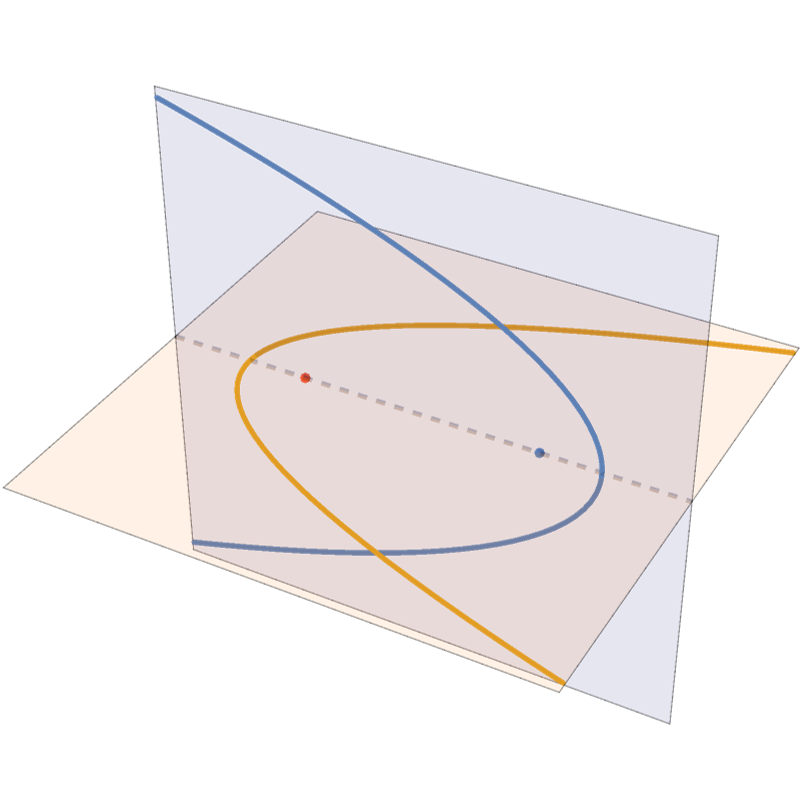}
        \caption{CPO Parabolas.}
    \end{subfigure}
    \caption{Congruent and perpendicularly opposite (CPO) parabolas.}
    \label{fig:TwoParabolas}
\end{figure}

For $a,m> 0$ and $b\neq 0$, such that $m\neq 1$, we consider the configuration 
\begin{align}\label{FirstConic}
C_1&=\left\{(x,y,z)\in \R^3\ :\ mx^2+y^2=a, \ z=0\right\} \quad \text{ and } \\
C_2&=\left\{(x,y,z)\in \R^3\ :\ \frac{m}{m-1}x^2+z^2=b, \ y=0\right\}. \label{SecondConic}
\end{align}
We note that $C_1$ is an ellipse for every $m>0$.
On the other hand, $C_2$ is an ellipse when $m>1$ and a hyperbola when $0<m<1$. 
We ignore the case where $b<0$ and $m>1$, since then $C_2$ is empty. 
Two curves \emph{match} if they are obtained by rotating and translating the configuration of \eqref{FirstConic} and \eqref{SecondConic}, for some values of $a,b,m$.
That is, two ellipses may match (Figure \ref{fig:EllipsesHyperbolas}(a)) and so do one ellipse and one hyperbola (Figure \ref{fig:EllipsesHyperbolas}(b)).
We can scale  \eqref{FirstConic} and \eqref{SecondConic} by changing the values of $a,b,m$.

 \begin{figure}[ht]
    \centering
    \begin{subfigure}[b]{0.27\textwidth}
    \centering
        \includegraphics[width=0.97\textwidth]{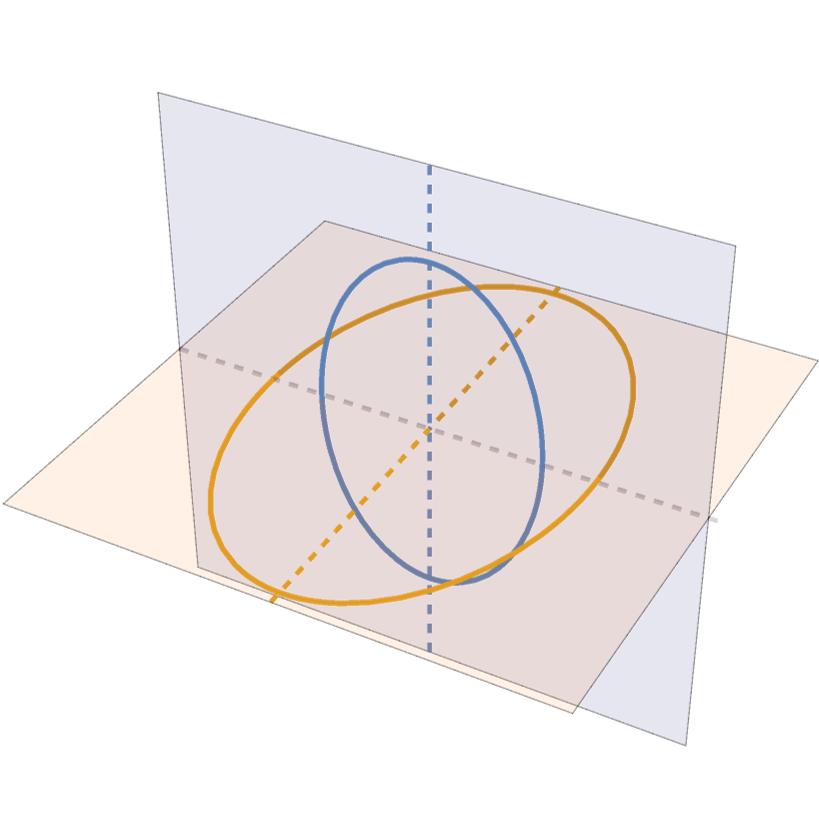}
        \caption{Two Ellipses.}
    \end{subfigure}
    \hspace{1cm}
    \begin{subfigure}[b]{0.25\textwidth}
        \centering
        \includegraphics[width=\textwidth]{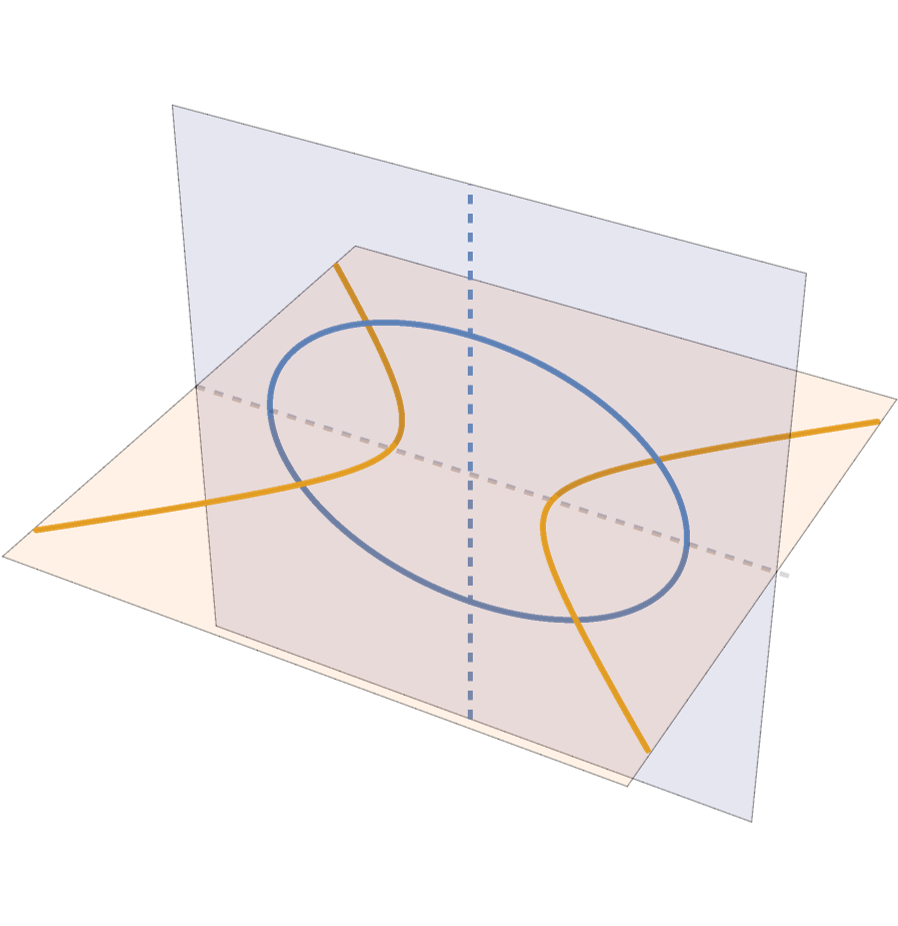}
        \caption{Ellipse and Hyperbola.}
    \end{subfigure}
    \caption{Matching curves.}
    \label{fig:EllipsesHyperbolas}
\end{figure}
\begin{theorem}[Few distances between conics]\label{th:FewDistancesBetweenConics}
Let $C_1,C_2\subset \R^3$ be non-degenerate conic sections. \\[2mm]
(a) Assume that $C_1$ and $C_2$ are either perpendicular circles, aligned circles, CPO parabolas, or matching curves. Then $(C_1,C_2)$ spans few distances. \\[2mm]
(b) If $(C_1,C_2)$ is not one of the configurations from part (a), then $(C_1,C_2)$ spans many distances.
\end{theorem}

We note that, as in the cases of $\R^2$ and circles in $\R^3$, every conic  configuration spans either few or many distinct distances. 

Theorem \ref{th:FewDistancesBetweenConics} does not include the case where at least one of the curves is a line. This case is significantly simpler to study, as shown in the proof of the following result.
Throughout this work, the word \emph{cylinder} refers specifically to right circular cylinders. 
That is, cylinders are the surfaces that are obtained by translating, rotating, and scaling $\vb(x^2+y^2-1)$. 
The \emph{axis} of a cylinder $S$ is the common axis of all the circles that are contained in $S$.

\begin{lemma}\label{le:LineAlgebraic}
Let $\ell \subset \R^3$ be a line and let $C\subset \R^3$ be a curve. \\[2mm]
(a) If $C$ is contained in a plane 
orthogonal to $\ell$ or in a cylinder centered around $\ell$, then $(C,\ell)$ spans few distances. \\[2mm]
(b) If $(C,\ell)$ is not one of the configurations from part (a), then $(C,\ell)$ spans many distances.
\end{lemma}

For another interesting variant where one point set is on a line, see Bruner and Sharir \cite{BS16}.

\parag{Our second result: curves on perpendicular planes.} 
We say that a curve in $C\subset \R^3$ is \emph{planar} if there exists a plane that contains $C$. 
Theorem \ref{th:PachDeZeeuw}
classifies the pairs of planar curves that span few distances, when the containing planes are parallel. In particular, the only configurations that span few distances in this case are parallel lines, orthogonal lines, and aligned circles. See Lemma \ref{le:ParallelPlanes} below. 

Our second result classifies all pairs of planar curves that span few distances when the containing planes are perpendicular. 

\begin{theorem}[Few distances between perpendicular planes]\label{th:PerpendicularPlanesFewDistances}
Let $C_1$ and $C_2$ be curves that are contained in perpendicular planes in $\R^3$. \\[2mm]
(a) Assume that $C_1$ and $C_2$ are one of the following: 
\begin{itemize}[noitemsep,topsep=0pt]
    \item a line $\ell$ and a curve contained in a plane orthogonal to $\ell$, 
    \item parallel lines, 
    \item a line $\ell$ and an ellipse contained in a cylinder centered around $\ell$, 
    \item CPO parabolas, 
    \item matching curves, 
    \item perpendicular circles.
\end{itemize}
Then $(C_1,C_2)$ spans few distances. \\[2mm]
(b) If $(C_1,C_2)$ is not one of the configuration from part (a), then $(C_1,C_2)$ spans many distances.
\end{theorem}

The proof of Theorem \ref{th:PerpendicularPlanesFewDistances} starts by studying sets that are not necessarily algebraic (see Theorem \ref{th:PerpendicularPlanesSpecialForm}).
This part of the proof leads to a surprising non-algebraic configuration that spans few distances.
We now describe this configuration. 

\begin{figure}[htp]     \centering
    \includegraphics[width=5cm]{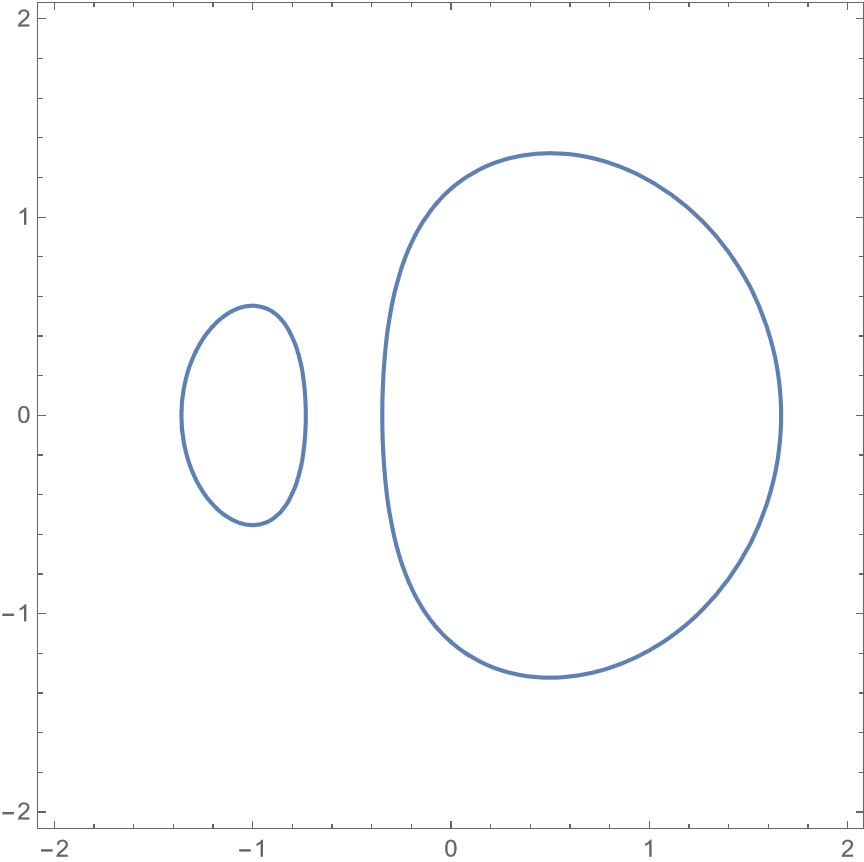}
    \caption{The log-circle $x^2+y^2=2+\ln|x+\frac{1}{2}|$.}\label{fi:LogCircle}
\end{figure}

For $A,B,D\in \R$, we consider the set in $\R^3$ that is defined by 
\begin{equation}\label{eq:LogCircle}
x^2+y^2=D+A\ln|x-B| \quad \text{ and } \quad z=0.
\end{equation}
We define \emph{log-circles} to be the sets that are obtained by translating 
and rotating the above set, with any $A,B,D\in \R$.
For example, see Figure \ref{fi:LogCircle}.
By setting $A=0$, we see that the family of log-circles includes all standard circles. 

For $A,B,D,D'\in \R$, consider the two log-circles 
\begin{align}
&\left\{(x,y,0)\in \R^3\ :\ (x-B)^2+y^2=D+A\ln|x| \right\}, \nonumber \\[2mm]
&\left\{(x,0,z)\in \R^3\ :\ x^2+z^2=D'+A\ln|x-B|\right\}. \label{eq:MatchingLogCirclesDef}
\end{align}

\begin{figure}[htp]     \centering
    \includegraphics[width=5cm]{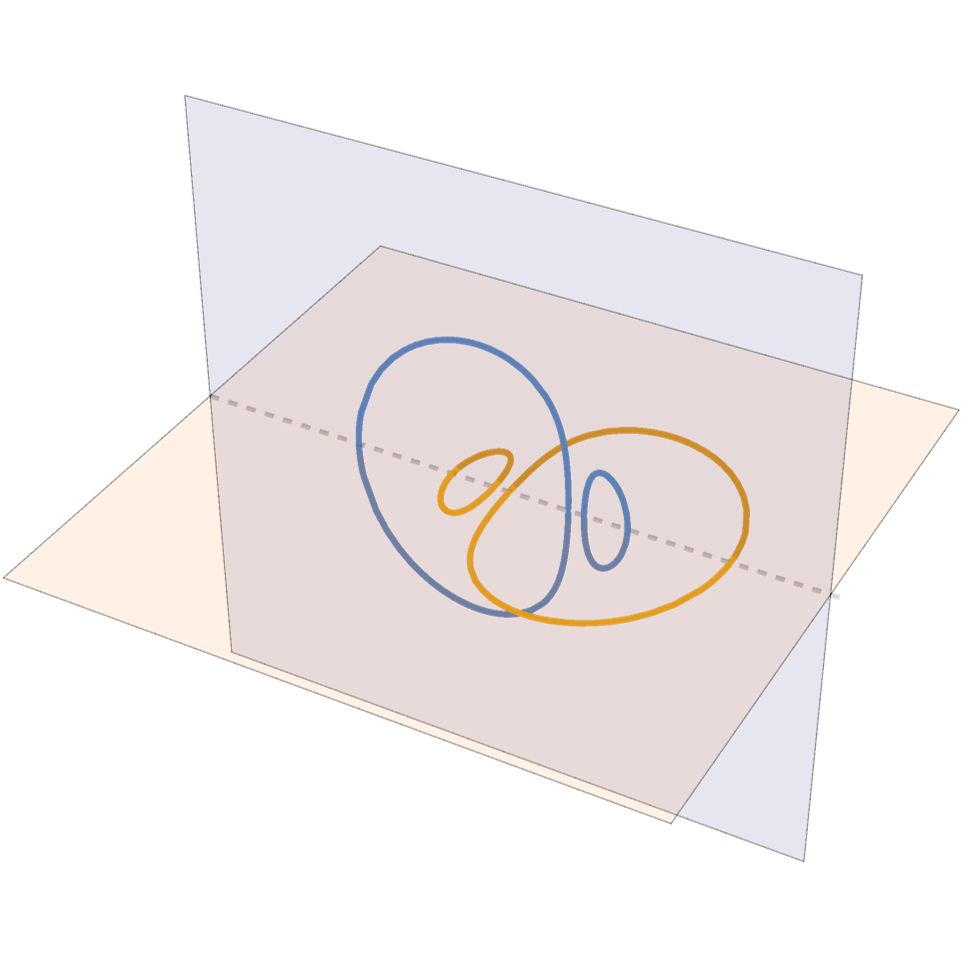}
    \caption{Matching Log-Circles.}\label{fi:MatchingLogCircles}
\end{figure}

The first log-circle is contained in the $xy$-plane and the second is contained in the $xz$-plane.
That is, these are one-dimensional sets that are contained in perpendicular planes.
We say that two log-circles \emph{match} if they can be obtained by rotating, translating, and scaling the above configuration, for some $A,B,D,D'\in \R$.
See Figure \ref{fi:MatchingLogCircles}. This is a non-algebraic generalization of aligned circles. 

\begin{theorem}\label{th:LogCirclesFewDistances}
Let $C_1$ and $C_2$ be matching log-circles. 
Then $(C_1,C_2)$ spans few distances.
\end{theorem}

\parag{Our technique.}
Our analysis begins with the proof of Mathialagan and Sheffer \cite{MS20}, but significantly generalizes this proof. We now briefly describe one of the main tools of our analysis.

We say that a set $S\subset \R^3$ is an \emph{arc} if there exists a smooth bijective map $\gamma:(0,1)\to S$. 
We also say that $\gamma$ is a \emph{parameterization} of $S$.
For more information, such as the definition of smoothness, see Section \ref{sec:prelim}.
We extend the properties of spanning many distances and spanning few distances to pairs of arcs in the straightforward way. 
That is, instead of considering point sets of two curves, we consider point sets on two arcs.

Consider arcs $A_1,A_2\subset \R^3$ with respetive parameterizations $\gamma_1,\gamma_2$.
For $s\in (0,1)$, we denote the three coordinates of $\gamma_i(s)$ as $\gamma_{i,x}(s)$, $\gamma_{i,y}(s)$, and $\gamma_{i,z}(s)$. 
The \emph{distance function} of $\gamma_1$ and $\gamma_2$ is 
\begin{equation} \label{eq:DistanceFunc}
\rho(s,t)=(\gamma_{1,x}(s)-\gamma_{2,x}(t))^2
+(\gamma_{1,y}(s)-\gamma_{2,y}(t))^2
+(\gamma_{1,z}(s)-\gamma_{2,z}(t))^2.
\end{equation}
The distance function  receives a point from each arc, using the parameterizations $\gamma_1$ and $\gamma_2$. 
The function returns the square of the distance between those two points.

A function $f:\R^2\to \R$ \emph{is of a special form} if there exist open intervals $I_1,I_2,I_3\subset \R$ and smooth functions $\phi_1:I_1\to \R, \phi_2:I_2\to \R, \phi_3:I_3\to \R$ that satisfy the following: For every $s\in I_1$ and $t\in I_2$, we have that $\phi_1(s)+\phi_2(t)\in I_3$ and that
\begin{equation} \label{eq:SpecialForm}
f(s,t)=\phi_3(\phi_1(s)+\phi_2(t)).
\end{equation}

The following lemma is a generalization of a lemma of Mathialagan and Sheffer \cite{MS20}, which in turn adapts a result of Raz \cite{Raz20}. 

\begin{lemma}\label{le:SpecialFormOrExpands}
Let $A_1,A_2\subset \R^3$ be arcs with respective parametrizations $\gamma_1,\gamma_2$. 
Let $\rho(s,t)$ be the distance function of $\gamma_1$ and $\gamma_2$. 
Then at least one of the following holds.\\[2mm]
(a) The configuration $(A_1,A_2)$ spans many distances. \\[2mm]
(b) The function $\rho(s,t)$ is of a special form. 
\end{lemma}

After proving Lemma \ref{le:SpecialFormOrExpands}, it remains to study when the distance function $\rho(s,t)$ is of a special form. 
To do that, we rely on the following \emph{derivative test}.
Here $\rho_s$ and $\rho_t$ are the first partial derivative of $\rho$ according to $s$ and $t$, respectively.

\begin{lemma}[Derivative test]\label{le:DerivativeTest}
Let $\rho:(0,1)^2\to \R$ be a smooth function of a special form, such that $\rho_s\not\equiv 0$ and $\rho_t\not\equiv 0$ in every open neighborhood in $\R^2$. Then there exists an open neighborhood in $\R^2$ in which 
\[
\frac{\partial^2 (\ln |\rho_s/\rho_t|)}{\partial s \partial t}\equiv 0.
\]
\end{lemma}

In our proofs, we study the expression from Lemma \ref{le:DerivativeTest}, when $\rho(s,t)$ is the distance fuction from \eqref{eq:DistanceFunc}. By using Wolfram Mathematica \cite{Wolfram}, we show that this expression is not identically zero in any open neighborhood.
Then the contrapositive of Lemma \ref{le:DerivativeTest} implies that $\rho(s,t)$ is not of a special form. We must thus be in case (a) of Lemma \ref{le:SpecialFormOrExpands}.

We also derive the following useful result. 

\begin{theorem}\label{th:FewOrMany}
Every pair of curves in $\R^3$ spans either few or many distances.
\end{theorem}

\parag{Open problems.} 
One main remaining open problem is to characterize all pairs of curves in $\R^3$ that span few distances.
With the above in mind, we propose the following question:
\begin{quote}
\emph{True or false? If two planar curves in $\R^3$ span few distances, then the planes that contain the curves are either parallel or perpendicular.}
\end{quote}
If the above statement is true, then Theorem \ref{th:PerpendicularPlanesFewDistances} and Lemma \ref{le:ParallelPlanes} characterize all pairs of \emph{planar} curves in $\R^3$ that span few distances.

Recently, Solymosi and Zahl \cite{SolyZhal22} improved the bound for the number of distinct distances between two lines from $\Omega(\min\{m^{2/3}n^{2/3},m^2,n^2\})$ to $\Omega(\min\{m^{3/4}n^{3/4},m^2,n^2\})$.
It seems plausible that the bounds of Theorem \ref{th:PachDeZeeuw}, Theorem \ref{th:ShefferMathialagan}, and the current work could also be improved in a similar way. 
Such a potential improvement would not replace the proofs of the current work. 
Instead, it would be added on top of our proofs, to amplify the bounds that are obtained. 

We also suggest studying  distinct distances in $\R^3$ between points on two curves that are not necessarily algebraic. 
Theorem \ref{th:PerpendicularPlanesSpecialForm} states that, when the two curves are contained in perpendicular planes, the only cases with few distances and a curve that is not algebraic are \begin{itemize}[noitemsep,topsep=0pt]
    \item one of the curves is a line and the other is contained in a parallel or perpendicular plane. 
    \item matching log-circles. 
    \item periodic functions that involve a sine or a cosine.
\end{itemize}
The third bullet point is not covered by Theorem \ref{th:PerpendicularPlanesSpecialForm}, since a periodic curve may not be the union of finitely many arcs and points.
In other words, problem is
\begin{quote}
\emph{Are there pairs of curves in $\R^3$ that span few distances, are not contained in perpendicular plane, not periodic, not lines, and at least one is not algebraic?.}
\end{quote}

\parag{Paper structure.}
Section \ref{sec:prelim} is a brief survey of basic real algebraic geometry that is used throughout the paper. 
Section \ref{sec:FirstBounds} contains two first distinct distances results which are easy to prove. 
We rely on these results in the following sections.
In Section \ref{sec:SpecialForm}, we study functions of special form, which are also used in the following sections.
In Section \ref{sec:Configurations}, we study configurations with few distinct distances. 
In Section \ref{sec:PerpendicularPlanes}, we prove Theorem \ref{th:PerpendicularPlanesFewDistances}.
In Section \ref{sec:FewDistancesBetweenConics}, we prove Theorem \ref{th:FewDistancesBetweenConics}.
Finally, in appendices \ref{app:Wolfram} and \ref{app:Wolfram2}, we go over Mathematica computations that are required for the proof of Theorem \ref{th:FewDistancesBetweenConics}. 

\parag{Acknowledgements.} We are grateful to Surya Mathialagan for the helpful conversations and support. 
We also wish to thank the mentors and staff of the Polymath Jr program, for making this research project possible. 

\section{Real algebraic geometry preliminaries} \label{sec:prelim}

In this section, we briefly survey notation and results from real algebraic geometry. 
The reader might wish to skim this section and return to it as necessary. 

For references and more information, see for example \cite{BCR98,Harris92}.
For polynomials $f_1,\ldots,f_k\in \R[x_1,\ldots,x_d]$, the \emph{variety} defined by $f_1,\ldots,f_k$ is
\[ \mathbf{V}(f_1,\ldots,f_k) = \left\{p\in \R^d\ :\ f_1(p)=f_2(p) = \cdots = f_k(p)=0 \right\}. \]
We say that a set $U \subset \R^d$ is a variety if there exist $f_1,\ldots,f_k\in \R[x_1,\ldots,x_d]$ such that $U = \vb(f_1,\ldots,f_k)$.
While not true over some other fields, in $\R^d$ every variety can be defined using a single polynomial. 
A variety $U$ is \emph{irreducible} if there do not exist two varieties $U_1,U_2\subsetneq U$ such that $U=U_1 \cup U_2$.
The \emph{dimension} of an irreducible variety $U$, denoted $\dim U$, is the largest integer $d_U$ for which there exist non-empty irreducible varieties $U_0,U_1,\ldots,U_{d_U}$ such that
\[
U_0 \subsetneq U_1 \subsetneq \ldots \subsetneq U_{d_U}=U. 
\]

We now mention two classic results about intersections of varieties in $\R^d$.

\begin{theorem}[Bezout's theorem]\label{th:BezoutTheorem}
Let $f$ and $g$ be polynomials in $\R[x,y]$ of degrees $k_f$ and $k_g$, respectively. If $f$ and $g$ do not have common factors then $\mathbf{V}(f)\cap \mathbf{V}(g)$ consists of at most $k_f \cdot k_g$ points.
\end{theorem}

\begin{theorem}[Milnor--Thom \cite{Milnor64,Thom65}]\label{th:MilnorThom}
Let $f_1,...,f_m\in \R[x_1,...,x_d]$ be of degree at most $k$. Then the number of connected components of $\mathbf{V}(f_1,...,f_m)$ is at most
\[
k(2k-1)^d.
\]
\end{theorem}

We observe the following simple corollary of Theorem \ref{th:MilnorThom}.

\begin{corollary}\label{co:FiniteOrContained}
Let $C\subset \R^3$ be a curve and let $U\subset \R^3$ be a variety. If $C\cap X$ is infinite then $C\subset X$.
\end{corollary}
\begin{proof}
We set $W = C\cap U$ and assume that $W$ is infinite. By definition, $W$ is a variety. 
If $\dim W = 0$, then this variety has infinitely many connected components. 
Since this contradicts Theorem \ref{th:MilnorThom}, we have that $\dim W \ge 1$. 
Since $C$ is irreducible and one-dimensional, we conclude that $C = W \subset X$.
\end{proof}

Unlike in $\C^d$, there are several non-equivalent definitions for the degree of a variety in $\R^d$. 
To avoid this issue, we say that the \emph{complexity} of a variety $U\subset \R^d$ is the minimum integer $D$ that satisfies the following: 
There exist $k\le D$ polynomials $f_1,\ldots,f_k\in \R[x_1,\ldots,x_d]$, each of degree at most $D$, such that $U=\mathbf{V}(f_1,\ldots,f_k)$. In the past decade, the use of complexity is becoming more common. For example, see \cite{BGT11,SolyTao12}.

For the following result, see for example \cite[Chapter 4]{Sheffer22}. 
In the $O_{k,d}(\cdot)$-notation, the hidden constant may depend on $d$ and $k$.

\begin{theorem} \label{th:components}
Let $U\subset \R^d$ be a variety of complexity $k$.
Then $U$ is the union of $O_{k,d}(1)$ irreducible varieties. 
\end{theorem}

We define a \emph{curve} to be an irreducible constant-complexity variety of dimension one. 
In Theorem \ref{th:FewDistancesBetweenConics} and Theorem \ref{th:PerpendicularPlanesFewDistances}, we implicitly consider sets of $m$ and $n$ points that determine $\Omega(\min\{m^{2/3}n^{2/3},m^2,n^2\})$ distinct distances.
In this context, the constant-complexity of the curves means that $m$ and $n$ may be assumed to be arbitrarily large with respect to the complexities of the curves. The hidden constant in the $\Omega(\cdot)$-notation may depend on the complexities of the curves.

Let $S\subset \R^d$. 
The \emph{Zariski closure} of $S$, denoted $\overline{S}$, is the smallest variety that contains $S$.
Specifically, every variety that contains $S$ also contains $\overline{S}$.
A set $S\subset \R^d$ is  \emph{semi-algebraic} if there exists a boolean function $\Phi(y_1,\ldots,y_t)$ and polynomials $f_1,\ldots,f_t\in \R[x_1,\ldots,x_d]$ such that
\[ p\in S \qquad \text{if and only if } \qquad \Phi(f_1(p)\ge 0,\ldots,f_t(p)\ge 0)=1. \]
The dimension of $S$ is $\dim \overline{S}$.

A \emph{standard projection} is a linear map $\pi:\R^n\to \R^m$ that keeps the first $m$ out of $n$ coordinates.
While the projection of a variety might not be a variety, it has other nice properties. 

\begin{theorem}\label{th:TarskiSeidenberg}
Let $U\subset \R^n$ be a variety of dimension $d$ and let $\pi:\R^n\to \R^m$ be a standard projection.   Then \\[2mm] 
(a) The projection $\pi(U)$ is a semi-algebraic set of dimension at most $d$. \\[2mm]
(b)  Let $d' = \dim \pi(U)$. Then there exists a variety $W \subset\R^m$ of dimension smaller than $d'$ that satisfies the following: 
For every $p \in \pi(U) \setminus W$, the set $\pi^{-1}(p)\cap U$ is  semi-algebraic of dimension at most $d - d'$.
\end{theorem}
For part (a), see for example  \cite[Section 11.3]{BPCR07}.
For part (b), see for example \cite[Section 7.1]{BES19}.

\ignore{We prove the contrapositive of part (c). 
That is, we assume that $\overline{\pi(U)}$ is reducible and show that $U$ is also reducible. 
By the assumption, there exist varieties $A,B\subsetneq \overline{\pi(U)}$ such that $A\cup B=\overline{\pi(U)}.$ 
We have that 
\[
U\subset\pi^{-1}(\overline{\pi(U)})=\pi^{-1}(A)\cup \pi^{-1}(B).
\]

By the definition of Zariski closure, we have that $\pi(U) \not\subset A$ and that $\pi(U) \not\subset B$.
This in turn implies that $U \not\subset \pi^{-1}(A)$ and that $U \not\subset \pi^{-1}(B)$. 
We set $U_1=\pi^{-1}(A)$ and $U_2=\pi^{-1}(B)$.
We note that $U_1\subsetneq U$, that $U_2\subsetneq U$, and that $U_1 \cup U_2 = U$.
Thus, $U$ is reducible.}

We say that a function $f:\R^d \to \R$ is \emph{smooth} if all its partial derivative of all orders exist.
We think of a function $g:\R^d \to \R^e$ as $e$ separate functions from $\R^d$ to $\R$ and say that $g$ is \emph{smooth} if each of these $e$ separate components is smooth. 

We recall that $S\subset \R^3$ is an \emph{arc} if there exists a smooth map $\gamma : (0,1) \to S$. We say that
$\gamma$ is a \emph{parameterization} of $S$. 
For the following lemma, see for example \cite{Raz20}.\footnote{We do not fully describe the simple proof of this lemma, since it would require many additional definitions. Briefly: $C$ has a finite number of singular points. After removing these points, each connected component is diffeomorphic either to $\R$ or to  $S^1$. Similarly, an arc is more commonly known as a connected one-dimensional smooth manifold.}

\begin{lemma}\label{le:SmoothParametrization}
Let $C\subset \R^n$ be a curve. Then $C$ is the union of finitely many arcs and points. 
\end{lemma}

\section{First distinct distances bounds}\label{sec:FirstBounds}

In this section, we study two first distinct distances results. 
These have shorter proofs and do not rely on involved tools such as Lemma \ref{le:SpecialFormOrExpands}.

\begin{lemma}\label{le:ParallelPlanes}
Let $C_1,C_2\subset \R^3$ be planar curves that are contained in parallel planes. \\[2mm]
(a) If $C_1$ and $C_2$ are parallel lines, orthogonal lines, or aligned circles, then $(C_1,C_2)$ spans few distances. \\[2mm] 
(b) If $(C_1,C_2)$ is not one of the configurations from part (a), then $(C_1,C_2)$ spans many distances.
\end{lemma}
\begin{proof}
(a) To see that parallel lines, orthogonal lines, and aligned circles span few distances, see for example \cite{MS20}.

(b) Let $H_1$ and $H_2$ be the planes that contain $C_1$ and $C_2$, respectively.
We rotate and translate $\R^3$ so that $H_1$ becomes the plane $\mathbf{V}(z)$.
Then there exists $d\in \R$ such that $H_2 = \mathbf{V}(z-d)$.
Note that rotations and translations do not affect distances. 

We consider the points 
$a=(a_x,a_y,0)\in C_1$ and $b=(b_x,b_y,d)\in C_2$. 
We also consider the standard projection $\pi(x,y,z)=(x,y)$.
The square of the distance between $a$ and $b$ equals to the square of the distance between $\pi(a)$ and $\pi(b)$ plus $d^2$. 
Since the squares of all the distances change by the same amount, we get that $D(\pts_1,\pts_2) = D(\pi(\pts_1),\pi(\pts_2))$.

By Theorem \ref{th:PachDeZeeuw}, $(\pi(C_1),\pi(C_2))$ spans few distances if they are parallel lines, orthogonal lines, or concentric circles. 
Otherwise, $(\pi(C_1),\pi(C_2))$ spans many distances.
If $\pi(C_1)$ and $\pi(C_2)$ are parallel or orthogonal lines, so are $C_1$ and $C_2$.
If $\pi(C_1)$ and $\pi(C_2)$ are concentric circles then $C_1$ and $C_2$ are aligned circles.
\end{proof}

We next prove Lemma \ref{le:LineAlgebraic}. We first recall the statement of this lemma. 
\vspace{2mm}

\noindent {\bf Lemma \ref{le:LineAlgebraic}.}
\emph{Let $\ell \subset \R^3$ be a line and let $C\subset \R^3$ be a curve. \\[2mm]
(a) If $C$ is contained in a plane orthogonal to $\ell$ or in a cylinder centered around $\ell$, then $(C,\ell)$ spans few distances. \\[2mm]
(b) If $(C,\ell)$ is not one of the configurations from part (a), then $(C,\ell)$ spans many distances.}
\begin{proof}
For part (a), see Lemmas \ref{le:LineCylinderFewDistances} and \ref{le:LineOrthogonalPlaneFewDistances} below. 
For part (b), we assume that $C$ is not contained in a plane orthogonal to $\ell$ or in a cylinder centered around $\ell$.
By rotating $\R^3$, we may assume that $\ell$ is the $z$-axis, without changing any distances.
We consider a polynomial $f\in\R[x,y,z]$ such that $\mathbf{V}(f) = C$. 

We consider a point $p\in \R^3$ and note that there is one circle $C_p$ that contains $p$ and has axis $\ell$. Moving $p$ along $C_p$ does not change the distances between $p$ and the points of $\ell$.
In particular, we may replace $p=(p_x,p_y,p_z)$ with the point $\left(\sqrt{p_x^2+p_y^2},0,p_z\right)$.
Thus, the distances between $\ell$ and $C$ are the same as the distances between $\ell$ and the set
\[ C' = \{(\sqrt{x^2+y^2},0,z)\ :\ f(x,y,z)=0 \}. \]
Since $\ell$ and $C'$ are contained in the plane $\mathbf{V}(y)$, we wish to apply Theorem \ref{th:PachDeZeeuw}. 
However, $C'$ may not be a curve. 

We denote the coordinates of $\R^4$ as $X,Y,Z,W$ and consider the variety
\[ U=\mathbf{V}(W^2-X^2-Y^2,f(X,Y,Z)) \subset \R^4. \]
Since $f$ defines a curve in $\R^3$ and $X,Y$ determine $W$ up to a sign, we get that $\dim U =1$.

We define the standard projection $\pi_2(X,Y,Z,W)=(Z,W)$.
We note that a point $(x,0,z)$ is in $C'$ if and only if the points $(z,\pm x)$ are in 
 $\pi_2(U)$. Since $C$ is not contained in a plane orthogonal to $\ell$, Corollary \ref{co:FiniteOrContained} implies that $C$ intersects such a plane in $O(1)$ points.
Thus, every point of $\pi_2(U)$ has $O(1)$ points of $U$ projected to it. 
Theorem \ref{th:TarskiSeidenberg} implies that $\overline{C'}$ is a one-dimensional variety.
Since $\overline{C'}$ may be reducible, it may not be a curve.

Since $C$ contains $O(1)$ points from every plane orthogonal to $\ell$, we get that $\overline{C'}$ does not contain lines that are orthogonal to $\ell$. 
Since $C$ is not contained in cylinders that are centered at $\ell$, Corollary \ref{co:FiniteOrContained} implies that $C$ intersects such a cylinder in $O(1)$ points.
This in turn implies that $\overline{C'}$ does not contain lines that are parallel to $\ell$.

Consider a set $\pts$ of $n$ points on $C$. 
Let $\pts'$ be the set of points of $\pts$ after taking them to $C'$ as described above. 
Since $O(1)$ points of $C$ are taken to the same point of $C'$, we have that $|\pts'|=\Theta(n)$.
Theorem \ref{th:components} implies that $C'$ consists of finitely many components. 
Thus, there exists an irreducible component of $C'$ that contains $\Theta(n)$ points of $\pts'$.
We complete the proof by separately applying Theorem \ref{th:PachDeZeeuw} with $\ell$ and each of the one-dimensional irreducible components of $C'$.

\end{proof}

\section{Functions of a special form}\label{sec:SpecialForm}

In this section, we study functions of a special form, as defined in \eqref{eq:SpecialForm}.
We begin by proving Lemma \ref{le:DerivativeTest}, which provides a test for checking whether a function is not of a special form.
We first recall the statement of this lemma, where $\rho_s$ is the first partial derivative of $\rho$ with respect to $s$. 
\vspace{2mm}

\noindent {\bf Lemma \ref{le:DerivativeTest}.}
\emph{Let $\rho:(0,1)^2\to \R$ be a smooth function of a special form, such that $\rho_s\not\equiv 0$ and $\rho_t\not\equiv 0$ in every open neighborhood in $\R^2$. Then there exists an open neighborhood in $\R^2$ in which }
\[
\frac{\partial^2 (\ln |\rho_s/\rho_t|)}{\partial s \partial t}\equiv 0.
\]
\begin{proof}
By definition, there exist open intervals $I_1,I_2,I_3$ and smooth $\phi_1:I_1\to \R,\phi_2:I_2\to \R$, $\phi_3:I_3\to \R$, such that every $(s,t)\in I_1\times I_2$ satisfies 
\[
\rho(s,t)=\phi_3(\phi_1(s)+\phi_2(t)).
\]

By the assumptions on $\rho_s$ and $\rho_t$, there exists an open set $U \subset I_1\times I_2$ in which $\rho_s$ and $\rho_t$ are never zero. 
In $U$, we have that
\begin{align*}
\rho_s
&=\frac{\partial [\phi_3(\phi_1(s)+\phi_2(t))]}{\partial s}
=\phi_3'(\phi_1(s)+\phi_2(t))\cdot \phi_1'(s), \\[2mm]
\rho_t
&=\frac{\partial [\phi_3(\phi_1(s)+\phi_2(t))]}{\partial t}
=\phi_3'(\phi_1(s)+\phi_2(t))\cdot \phi_2'(t).
\end{align*}
With these equations, we get that 
\begin{align*}
  \frac{\partial^2 (\ln |\rho_s/\rho_t|)}{\partial s \partial t} 
  =\frac{\partial^2 (\ln |\phi_1'(s)/\phi_2'(t)|)}{\partial s \partial t} 
  =\frac{\partial^2 (\ln |\phi_1'(s)|-\ln|\phi_2'(t)|)}{\partial s \partial t} 
  =0,
\end{align*}
for all $(s,t)\in U.$
\end{proof}

\ignore{We also require the following corollary of Lemma \ref{le:DerivativeTest}. 

\begin{corollary}\label{co:DerivativeTestClosedUnderComposition}
Consider intervals $I_1$ and $I_2$ and a smooth  $\rho:I_1\times I_2 \to \R$ that satisfies
\[
\frac{\partial^2 (\ln |\rho_s/\rho_t|)}{\partial s \partial t}\equiv 0.
\]
Consider smooth $f,g:\R \to \R$ and set $\widetilde{\rho}(s,t)=\rho(f(s),g(t))$. Then
\[
\frac{\partial^2 (\ln |\widetilde{\rho}_s/\widetilde{\rho}_t|)}{\partial s \partial t}\equiv 0.
\]
\end{corollary}
\begin{proof}
Let $\hat{\rho}(s,t)=\rho(f(s),t)$. 
Lemma \ref{le:DerivativeTest} implies that
\begin{align*}
\frac{\partial^2 (\ln \left|\hat{\rho}_s(s,t)/\hat{\rho}_t(s,t)\right|)}{\partial s \partial t}
&=\frac{\partial^2 (\ln |\rho_s(f(s),t) f'(s)/\rho_t(f(s),t)|)}{\partial s \partial t}\\[2mm]
&=\frac{\partial^2 \ln|f'(s)|}{\partial s \partial t}+\frac{\partial^2 (\ln |\rho_s(f(s),t)/\rho_t(f(s),t)|)}{\partial s \partial t} =0.
\end{align*}

To complete the proof, we repeat the above argument with $\widetilde{\rho}(s,t)=\hat{\rho}(s,g(t))$.
\end{proof} } 

We next use special forms to study configurations that span few distances.
Recall that the distance function of two arc parameterizations is defined in  \eqref{eq:DistanceFunc}.

\begin{lemma}[Special form implies few distances]\label{le:SpecialFormFewDistances}
Let $A_1,A_2\subset \R^3$ be arcs and let $\gamma_1,\gamma_2$ be parameterizations of those arcs, respectively. 
If the distance function of $\gamma_1$ and $\gamma_2$ is of a special form then $(A_1,A_2)$ spans few distances.
\end{lemma}
\begin{proof}
Let $\rho(s,t)$ be the distance function of $\gamma_1$ and $\gamma_2$. Let $I_1,I_2,I_3\subseteq (0,1)$ and $\phi_1,\phi_2,\phi_3$ be as in the definition of $\rho(s,t)$ having a special form.
By definition, $\phi_1(I_1)$ and $\phi(I_2)$ are nonempty intervals in $\R$. 
We first consider the case where $\phi_1(I_1)$ is a single point. 
In this case, the choice of point from $\gamma_1(I_1)$ does not affect the distance with any point of $\gamma_2(I_2)$. 
Thus, the number of distinct distances between $m$ points from $\gamma_1(I_1)$ and $n$ points from $\gamma_2(I_2)$ is at most $n$. 
A symmetric argument shows that, when $\phi_2(I_2)$ is a single point, there are at most $m$ distances. 

It remains to consider the case where both  $\phi_1(I_1)$ and $\phi_2(I_2)$ are intervals with infinitely many points. 
Let $p\in \phi_1(I_1)$ and $q\in \phi_2(I_2)$ be interior points of those intervals. 
Let $\eps>0$ be sufficiently small, as described below.
For $1\le i \le m$, we set $x_i=p+i\eps$.
For $1\le i \le n$, we set $x'_i=q+i\eps$.
When $\eps$ is sufficiently small, we have that $x_1,\ldots,x_m\in \phi_1(I_1)$ and $x'_1,\ldots,x'_n\in \phi_2(I_2)$. 
Let $s_i\in I_1$ satisfy $\phi_1(s_i)=x_i$ and let $t_j\in I_2$ satisfy $\phi_2(t_j)=x'_j$.
Then
\begin{align*}
\rho(s_i,t_j) =\phi_3(\phi_1(s_i)_+\phi_2(t_j)) =\phi_3(x_i+x'_j)
=\phi_3((x+x')+(i+j)\eps).
\end{align*}
We note that the expression $i+j$ has $m+n-1$ distinct values.
This implies that $\rho(s_i,t_j)$ has at most $m+n-1$ distinct values.
We conclude that the sets $\pts_1 = \{x_1,\ldots,x_m\} \subset A_1$ and $\pts_2 = \{x_1',\ldots,x_n'\}\subset A_2$ span $O(m+n)$ distinct distances.
\end{proof}

The converse of Lemma \ref{le:SpecialFormFewDistances} holds when the two arcs are contained in curves. 

\begin{lemma} \label{le:SpecialFormOrManyDist}
Let $A_1,A_2\subset \R^3$ be arcs and let $\gamma_1,\gamma_2$ be parameterizations of those arcs, respectively. 
Let $C_1$ and $C_2$ be curves that contain $A_1$ and $A_2$, respectively. 
If the distance function of $\gamma_1$ and $\gamma_2$ is not of a special form then $(A_1,A_2)$ spans many distances.
\end{lemma}

To prove Lemma \ref{le:SpecialFormOrManyDist}, we rely on the following result of Raz, Sharir, de Zeeuw \cite[Sections 2.1 and 2.3]{RSdZ16}.
See also Raz \cite[Lemma 2.4]{Raz20}.
\begin{theorem} \label{th:CartesianProduct}
Let $F \in \R[x, y, z]$ be a constant-degree irreducible polynomial, such that all first partial derivatives of $F$ are not identically zero.
Then at least one of the following two cases holds.\\[2mm]
(i) For all $A, B \subset \R$ with $|A| = m$ and $|B| = n$, we have that
\begin{multline*}
   \left|\left\{(a,a',b,b')\in A^2 \times B^2 \mid \exists c\in \R :  F(a,b,c)=F(a',b',c)=0 \right\}\right| = O\left(m^{4/3}n^{4/3}+m^2+n^2\right). 
\end{multline*}
(ii) There exists a one-dimensional variety $Z^* \subset \mathbf{V}(F)$ of degree $O(1)$ that satisfies the following.
    For every $p \in  \mathbf{V}(F) \setminus Z^*$, there exist open intervals $I_1, I_2, I_3 \subset \R$ and real-analytic functions $\varphi_1,\varphi_2,\varphi_3$ with analytic inverses such that $p \in I_1\times I_2\times I_3$ and for all $(x, y, z) \in I_1\times I_2\times I_3$ we have
    \[ (x, y, z) \in \mathbf{V}(F) \quad \text{ if and only if } \quad \varphi_1(x) + \varphi_2(y) + \varphi_3(z) = 0. \]
\end{theorem}

\begin{proof}[Proof of Lemma \ref{le:SpecialFormOrManyDist}.]
Assume that the distance function of $\gamma_1$ and $\gamma_2$ is not of a special form.
We rotate $\R^3$ so that $C_1$ and $C_2$ are not contained in planes that are parallel to the $yz$-plane.
A rotation does not change distances in $\R^3$.
This in turn implies that a rotation does not change whether a distance function has a special form or not.

Since every variety in $\R^d$ is the zero set of a single polynomial, there exist $f_1,f_2\in \R[x,y,z]$ such that $C_1=\mathbf{V}(f_1)$ and $C_2=\mathbf{V}(f_2)$. We consider points $p=(p_x,p_y,p_z)\in C_1$ and $q=(q_x,q_y,q_z)\in C_2$ at a distance of $\delta\in \R$ from each other. Setting $\Delta=\delta^2$ leads to
\begin{align*}
f_1(p_x,p_y,p_z) &= 0, \\
f_2(q_x,q_y,q_z) &=0, \\
(p_x-q_x)^2+(p_y-q_y)^2&+(p_z-q_z)^2 = \Delta.
\end{align*}

When considering $p_x,p_y,p_z,q_x,q_y,q_z,\Delta$ as the coordinates of a seven-dimensional space $\R^7$, the above system defines a variety $U\subset\R^7$. 
Since $f_1$ defines a curve, the first equation of the system defines a one-dimensional set of values for $p_x,p_y,p_z$.
Similarly, the second equation of the system defines a one-dimensional set of values for $q_x,q_y,q_z$. 
When the values of $p_x,p_y,p_z,q_x,q_y,q_z$ are fixed, the third equation uniquely defines $\Delta$. 
This implies that $\dim(U)=2$.

Let $\pi:\R^7 \to \R^3$ be the projection defined as
\[
\pi(p_x,p_y,p_z,q_x,q_y,q_z,\Delta) = (p_x,q_x,\Delta). 
\]
We claim that every point of $U_3$ has $O(1)$ points of $U$ projected to it. 
Indeed, recall that $C_1$ is not contained in planes parallel to the $yz$-plane. 
Theorem \ref{th:MilnorThom} implies that $C_1$ intersects such a plane in $O(1)$ points. 
That is, the number of points on $C_1$ that have the same $x$-coordinate is $O(1)$.
A symmetric argument holds for $C_2$.
Thus, every triple $(p_x,q_x,\Delta)\in U_3$ corresponds to $O(1)$ points of $U$.
Since $\dim U =2$, Theorem \ref{th:TarskiSeidenberg} implies that $U_3$ is a semi-algebraic set of dimension two. 
Let $F\in \R[p_x,q_x,\Delta]$ satisfy $\mathbf{V}(F) = \overline{U_3}$.

\parag{Studying $F$.} We now prove that $F$ involves all three variables $p_x,q_x,\Delta$. 
Recall that a fixed $p_x$ corresponds to $O(1)$ points $p\in C_1$ and a fixed $q_x$ corresponds to $O(1)$ points $q\in C_2$.
This implies that, for fixed $p_x,q_x$, there are $O(1)$ values for $\Delta$ that satisfy $(p_x,q_x,\Delta)\in U_3$. 
In other words, every line that is parallel to the $\Delta$-axis intersects $U_3$ in $O(1)$ points. 
Thus, $\Delta$ appears in the definition of $F$. 

If $C_1$ is a circle and $C_2$ is the axis of $C_1$ then $\rho(s,t)$ does not depend on $s$. This implies that $\rho(s,t)$ is of a special form. The same happens when switching the roles of $C_1$ and $C_2$. 
Since the distance function of $\gamma_1$ and $\gamma_2$ is not of a special form, $C_1$ and $C_2$ are not a circle and its axis. 

Consider fixed values for $p_x$ and $\Delta$, such that infinitely many values of $q_x$ satisfy $(p_x,q_x,\Delta)\in U_3$.
Since $p_x$ corresponds to $O(1)$ points $p\in C_1$, there exists such a point $p$ that is at distance $\sqrt{\Delta}$ from infinitely many points $q\in C_2$.
Let $S$ be the sphere of radius $\sqrt{\Delta}$ centered at $p$. 
The above implies that $C_2$ has an infinite intersection with $S$.
Corollary \ref{co:FiniteOrContained} implies that $C_2\subset S$. 
Assume that there exist $(p_x', \Delta') \neq (p_x,\Delta)$ such that infinitely many values of $q_x$ satisfy $(p_x',q_x,\Delta')\in U_3$.
Then there exists $p'$ with $x$-coordinate $p_x'$ such that the sphere $S'$ of radius $\sqrt{\Delta'}$ centered at $p'$ contains $C_2$. 
We have that $p_x\neq p_x'$, since otherwise $S\cap S' = \emptyset$. 
Thus, $C_2$ is the circle $S\cap S'$. 

Continuing the preceding paragraph, we next assume that there exist infinitely many pairs $(p_x,\Delta)$, each having infinitely values of $q_x$ that satisfy $(p_x,q_x,\Delta)\in U_3$.
As before, that means that $C_2$ is a circle and that every pair has $(p_x,\Delta)$ has a distinct $p_x$.
For a corresponding point $p\in C_1$ to be equidistant from infinitely many points of $C_2$, the point $p$ must be on the axis of $C_2$. 
That is, infinitely many points of $C_1$ are on the axis of $C_2$. 
Corollary \ref{co:FiniteOrContained} implies that $C_1$ is the axis of $C_2$. 
This is contradicts the above statement that $C_1$ and $C_2$ are not a circle and its axis. 
We conclude that the number of pairs $(p_x,\Delta)$ that satisfy the above is finite.
In other words, $U_3$ contains a $O(1)$ lines that are parallel to the $q_x$-axis. Thus, $q_x$ appears in the definition of $F$. 
A symmetric arguments holds for $p_x$.

Since all three coordinates participate in the definition of $F$, no first partial derivative of $F$ is identically zero. We may thus apply Theorem \ref{th:CartesianProduct} with $F$. 
We partition the remainder of the proof according to the case of the theorem that holds.

\parag{\textbf{Case (i).}} We first assume that case (i) of Theorem \ref{th:CartesianProduct} holds.
Let $\P_1\subset A_1$ and $\P_2\subset A_2$ satisfy $|\pts_1| = m$ and $|\pts_2|=n$. 
Let $\P_{1,x}$ and $\P_{2,x}$ be sets of the $x$-coordinates of the points of $\P_1$ and $\P_2$, respectively.
We set 
\begin{align*}
Q &= \left\{(a,a',b,b')\in \pts_1^2 \times \pts_2^2 \mid \exists \Delta\in \R : F(a_x,b_x,\Delta)=F(a_x',b_x',\Delta)=0 \right\},\\[2mm]
Q_x&=\left\{(a_x,a_x',b_x,b_x')\in \P_{1,x}^2 \times \P_{2,x}^2 \mid \exists \Delta\in \R :  F(a_x,b_x,\Delta)=F(a_x',b_x',\Delta)=0 \right\}.
\end{align*}
Case (i) of Theorem \ref{th:CartesianProduct} implies that
\[
|Q_x|=O\left(m^{4/3}n^{4/3}+m^2+n^2\right).
\]

Since every $x$-coordinate appears $O(1)$ times in $C_1$ and $C_2$, a quadruple $(a_x,a_x',b_x,b_x')\in Q_x$ corresponds to $O(1)$ quadruples of $Q$. 
This in turn implies that
\begin{equation} \label{eq:UpperQuadruples}
|Q|= O(|Q_x|) = O\left(m^{4/3}n^{4/3}+m^2+n^2\right).
\end{equation}

For $\delta\in \R$, let $m_{\delta}$ be the number of pairs in $\P_1\times \P_2$ that span the distance $\delta$.
The number of quadruples $(a,a',b,b')\in Q$ that satisfy $|ab|=|a'b'|=\delta$ is $m_{\delta}^2$. Thus $|Q|=\sum_{\delta} m_\delta^2.$ 
We also note that every pair of $\P_1\times\P_2$ contributes to exactly one $m_\delta$.
This implies that $\sum_\delta m_\delta = mn$.
Combining these observations with the Cauchy--Schwarz inequality leads to 
\[
|Q| =\sum_{\delta\in \R}m_\delta^2 
\ge \frac{(\sum_\delta m_\delta)^2}{D(\pts_1,\pts_2)} = \frac{m^2n^2}{D(\pts_1,\pts_2)}.
\]
Combining this with \eqref{eq:UpperQuadruples} gives
\[ 
\frac{m^2n^2}{D(\pts_1,\pts_2)} = O\left(m^{4/3}n^{4/3}+m^2+n^2\right), \quad \text{ or } \quad D(\pts_1,\pts_2)=\Omega\left(\min\left\{m^{2/3}n^{2/3},m^2,n^2\right\}\right).
\]
In other words, $A_1$ and $A_2$ span many distances.

\parag{\textbf{Case (ii).}} We now assume that case (ii) of Theorem \ref{th:CartesianProduct} holds.
We set
\[ 
Z = \left\{(\gamma_{1,x}(s),\gamma_{2,x}(t),\rho(s,t))\ : \ s,t\in (0,1) \right\}.
\]
By definition, $Z$ is contained in $U_3$.
It is a continuous two-dimensional semi-algebraic set. 
Intuitively, we can imagine $Z$ a patch of the surface $U_3$.

Let $Z^*$ be as described in case (ii) of Theorem \ref{th:CartesianProduct}. We consider $v\in Z\setminus Z^*$. Such a $v$ exists, since $Z$ is two-dimensional while $Z^*$ is one-dimensional. 
By Theorem \ref{th:CartesianProduct}, there exist open intervals $I_1,I_2,I_3\subset \R$ and real-analytic $\phi_1:I_1\to \R,$ $\phi_2:I_2\to \R, \phi_3: I_3 \to \R$ with analytic inverses such that $v\in I_1\times I_2\times I_3$ and every $(p_x,q_x,\Delta)\in I_1\times I_2\times I_3$ satisfies
\begin{equation*} 
(p_x, q_x, \Delta) \in \Vb(F) \quad \text{ if and only if } \quad \phi_1(p_x) + \phi_2(q_x) + \phi_3(\Delta) = 0.
\end{equation*}
Set $\phi_1'(p_x) = -\phi_1(p_x)$ and $\phi_2'(q_x) = -\phi_2(q_x)$.
Then, we can rearrange 
$\phi_1(p_x) + \phi_2(q_x) + \phi_3(\Delta) = 0$ as 
$\Delta=\phi_3^{-1}(\phi_1'(p_x)+\phi_2'(q_x))$.
Combining this with the above leads to 
\begin{equation}\label{eq:FirstIff}
(p_x, q_x, \Delta) \in \Vb(F) \quad \text{ if and only if } \quad  \Delta=\phi_3^{-1}(\phi_1'(p_x)+\phi_2'(q_x)).
\end{equation}

We write $v=(p_{x,0},q_{x,0},\Delta_0)$. 
Since $v\in Z$, there exist $s_0,t_0\in (0,1)$ that satisfy
\[
(\gamma_{1,x}(s_0),\gamma_{2,x}(t_0),\rho(s_0,t_0))=(p_{x,0},q_{x,0},\Delta_0).
\]
There exist open intervals $S\subset (0,1)$ and $T\subset (0,1)$ such that $s_0\in S$, $t_0\in T$,  $\gamma_{1,x}(S)\subset I_1$, $\gamma_{2,x}(T)\subset I_2$, and $\rho(S,T)\subset I_3.$ 
Consider $s\in S$ and $t\in T$ such that $\rho(s,t)=\Delta$. 
Then $(\gamma_{1,x}(s),\gamma_{2,x}(t),\Delta)\in \mathbf{V}(F)$.
Combining this with \eqref{eq:FirstIff} implies that 
\begin{equation*}
\rho(s,t)=\phi_3^{-1}\big(\phi_1'(\gamma_{1,x}(s)) + \phi_2'(\gamma_{2,x}(t))\big). 
\end{equation*}
Since this holds for every $s\in S$ and $t\in T$, 
we conclude that $\rho(s,t)$ is of a special form. 
This contradiction implies that Case (ii) cannot occur.
\end{proof}

By combining Lemma \ref{le:SpecialFormFewDistances} and Lemma \ref{le:SpecialFormOrManyDist}, we obtain Theorem  \ref{th:FewOrMany}.
We first recall the statement of this theorem.
\vspace{2mm}

\noindent {\bf Theorem \ref{th:FewOrMany}.}
\emph{Every pair of curves in $\R^3$ spans either few or many distances.}
\begin{proof}
Consider two curves $C_1,C_2\subset \R^3$. By Lemma \ref{le:SmoothParametrization}, each curve is the union of $O(1)$ arcs and points. We first assume that there exist two arcs, one from each curve, with parameterizations whose distance function is of a special form. 
In this case, Lemma \ref{le:SpecialFormFewDistances} states that the two arcs span few distances. This in turn implies that $(C_1,C_2)$ spans few distances.

We next assume that no two arcs of $C_1$ and $C_2$ have a distance function of a special form. Let $\pts_1 \subset C_1$ be a set of $m$ points and let $\pts_2\subset C_2$ be a set of $n$ points. 
Since $C_1$ is the union of $O(1)$ arcs and points, there exists an arc of $A_1\subset C_1$ that contains $\Theta(m)$ points of $\pts_1$. Similarly, there exists an arc of $A_2\subset C_2$ that contains $\Theta(n)$ points of $\pts_2$. 
Lemma \ref{le:SpecialFormOrExpands} states that %
\[ D(\pts_1,\pts_2)\ge D(\pts_1\cap A_1,\pts_2\cap A_2)=\Omega\left(\min\left\{m^{2/3}n^{2/3},m^2,n^2\right\}\right). \]
\end{proof}

\section{Configurations that span few distances}\label{sec:Configurations}

In this section, we study configurations that span few distances. To show that a configuration spans few distances, we show that the distance function is of a special form and then apply Lemma \ref{le:SpecialFormFewDistances}.
Recall that, in this work, we only consider right circular cylinders.

\begin{lemma}\label{le:LineCylinderFewDistances}
Let $\ell \subset \R^3$ be a line in $\R^3$ and let $C$
be a curve contained in a cylinder centered around $\ell$. Then $(\ell,C)$ spans few distances.
\end{lemma}
\begin{proof}
We translate and rotate the space so that $\ell$ becomes the $z$-axis. We then perform a uniform scaling of $\R^3$ so that $C$ is contained in a cylinder of radius one around $\ell$. These transformation do not change the number of distances between point sets. 
We first assume that all points of $C$ have the same $z$-coordinate. 
In this case, $C$ it must be a circle contained in a plane orthogonal to $\ell$. A fixed point of $\ell$ spans the same distance with every point of $C$. 
Thus, the number of distinct distances between $m$ points on $\ell$ and $n$ points on $C$ is at most $m$.
 
Next, we assume that $C$ is not a circle in a plane orthogonal to $\ell$. 
In this case, the set of $z$-coordinates of points of $C$ is an infinite closed interval $[z_1,z_2] \subset \R$. 
As a parameterization of $\ell$, we set $\gamma_1(s)=(0,0,s)$. 
For $t\in (z_1,z_2)$, the distance function between the point $(0,0,s)$ and a point of $C$ with $z$-coordinate $t$ is
\[
\rho(s,t)=1+(s-t)^2.
\]
Let $A$ be an arc of $C$ that is parameterized by its $z$-coordinate. 
Since $\rho(s,t)$ is of a special form, Lemma \ref{le:SpecialFormFewDistances} implies that $A$ and $\ell$ span few distances.
\end{proof}

In Lemma \ref{le:LineCylinderFewDistances}, it is not difficult to replace $C$ with any set that contains an arc. 
That is, Lemma \ref{le:LineCylinderFewDistances} can be extended to non-algebraic sets.

\begin{lemma}
\label{le:LineOrthogonalPlaneFewDistances}
Let $L\subset \R^3$ be a line and let $C$ be a curve contained in a plane orthogonal to $L$. Then $(L,C)$ spans few distances.
\end{lemma}
\begin{proof}
We rotate and translate the space so that $\ell$ becomes the $z$-axis and $C$ is contained in the $xy$-plane. 
These transformations do not change distances between points. 
We define the \emph{radial coordinate} of a point $(x,y,z)\in\R^3$ as $\sqrt{x^2+y^2}$.
If all the points of $C$ have the same radial coordinate then $C$ is a circle with axis $\ell$.  This case is covered by Lemma \ref{le:LineCylinderFewDistances}. 

Next, we assume that the points of $C$ have more than one radial coordinate. 
In this case, the set of radial coordinates is a closed infinite interval $[r_1,r_2]$.
As a parameterization of $\ell$, we set $\gamma_1(s)=(0,0,s)$. 
For $t\in (r_1,r_2)$, the distance function between the point $(0,0,s)$ and a point of $C$ with radial coordinate $t$ is
\[ \rho(s,t)=s^2+t^2. \]
Let $A$ be an arc of $C$ that is parameterized by its radial coordinate. 
Since $\rho(s,t)$ is of a special form, Lemma \ref{le:SpecialFormFewDistances} implies that $A$ and $\ell$ span few distances.
\end{proof}

As with Lemma \ref{le:LineCylinderFewDistances}, it is not difficult to extend Lemma \ref{le:LineOrthogonalPlaneFewDistances} to non-algebraic sets.

\begin{lemma}\label{le:AlignedCirclesFewDistances}
Let $C_1$ and $C_2$ be aligned circles. Then $(C_1,C_2)$ spans few distances.
\end{lemma}
\begin{proof}
We translate, rotate, and scale $\R^3$ so that the axis of both circles is the $z$-axis, $C_1$ is contained in the $xy$-plane, and $C_2$ is contained in the plane $\mathbf{V}(z-1)$.
Such transformation do not change the number of distances between two point sets. Then, $C_1$ can be parameterized as $\gamma_1(s)=(r_1\cos 2\pi s, r_1 \sin 2\pi s, 0)$ and $C_2$ as $\gamma_2(t)=(r_2 \cos 2\pi t, r_2 \sin 2\pi t, 1)$, where $s,t\in (0,1)$. The distance function of these parametrizations is
\begin{align*}
   \rho(s,t)
    &=(r_1 \cos 2\pi s-r_2 \cos 2\pi t)^2+(r_1 \sin 2\pi s - r_2 \sin 2\pi  t)^2+1\\
    &=r_1^2+r_2^2-2r_1r_2(\cos 2\pi s \cdot \cos 2\pi t + \sin 2\pi s\cdot \sin 2\pi t)+1\\
    &=(1+r_1^2+r_2^2)-2 r_1 r_2 \cos 2\pi (s-t).
\end{align*}
Since $\rho(s,t)$ is of a special form, Lemma \ref{le:SpecialFormFewDistances} implies that $C_1$ and $C_2$ span few distances.
\end{proof}

The case of perpendicular circles is subsumed by the case of log-circles below.
For explicit constructions of point sets on aligned and perpendicular circles, see Mathialagan and Sheffer \cite{MS20}.

\begin{lemma}\label{le:PerpendicularParabolasFewDistances}
Let $C_1$ and $C_2$ be CPO parabolas. Then $(C_1,C_2)$ spans few distances.
\end{lemma}
\begin{proof}
We translate, rotate, and perform a uniform scaling of $\R^3$ so that $C_1$ becomes the parabola $\Vb(y-x^2,z)$.
Such transformations do not affect the number of distinct distances between two sets. 
Since these are CPO parabolas, there exists $q\in \R$ such that $C_2 = \Vb(y-q+z^2,x)$.
We parameterize $C_1$ as
$\gamma_1(s) = (s,s^2,0)$ and $C_2$ as $\gamma_2(t) = (0,q-t^2,t)$.
Then, the distance function is 
\begin{align*}
\rho(s,t)=s^2+(s^2+t^2-q)^2+t^2=((s^2+t^2)-q)^2+(s^2+t^2).
\end{align*}
Since $\rho(s,t)$ is of a special form, Lemma \ref{le:SpecialFormFewDistances} implies that $C_1$ and $C_2$ span few distances.
\end{proof}

The expression \emph{matching curves} refers to two ellipses or to an ellipse and a hyperbola, as defined in \eqref{FirstConic} and \eqref{SecondConic}. In particular, matching log-circles are not curves by our definition.

\begin{lemma}\label{le:MatchingEllipseHyperbolaFewDistances}
Let $C_1$ and $C_2$ be matching curves. 
Then $(C_1,C_2)$ spans few distances.
\end{lemma}
\begin{proof}
By \eqref{FirstConic} and \eqref{SecondConic}, we can parameterize $C_1$ with $\gamma_1(s)=(s,\sqrt{a-ms^2},0)$ and $C_2$ with  $\gamma_2(t)=(t,0,\sqrt{b-mt^2/(m-1)})$.
When $0<m<1$, the distance function is
\begin{align*}
    \rho(s,t)
    &=(s-t)^2+(a-ms^2)+(b-mt^2/(m-1))\\
    &=(a+b)+(1-m)s^2+\frac{1}{1-m}t^2-2st\\
    &=(a+b)+\left(\sqrt{1-m}\cdot s-\sqrt{\frac{1}{1-m}}\cdot t\right)^2.
\end{align*}

When $m>1$, the last line of the above calculation becomes
\[ \rho(s,t) = (a+b)-\left(\sqrt{m-1}s+\sqrt{\frac{1}{m-1}}t\right)^2. \]

In either case, $\rho(s,t)$ is of a special form. 
Thus, Lemma \ref{le:SpecialFormFewDistances} implies that $C_1$ and $C_2$ span few distances.
\end{proof}

\parag{Log-circles.} 
Recall that log-circles are defined in \eqref{eq:LogCircle}. 
We define the \emph{positive arc} of a log-circle to be the subset of the log-circle that is defined by
\[ x^2+y^2=D+\ln(x-B). \]
The \emph{negative arc} is the subset defined by
\[ x^2+y^2=\ln(B-x). \]
Below we see that these are indeed arcs, since they have smooth parameterizations.
In Figure \ref{fig:LogCirclesCases}(a), the positive arc is the connected component on the right and the negative arc is the component on the left. 
It is possible for a log-circle to consist of only one arc. For example, the log-circle that is defined by $x^2+y^2=2+5\cdot \log |x+1|$ has an empty negative arc (Figure \ref{fig:LogCirclesCases}(b)).

 \begin{figure}[ht]
    \centering
    \begin{subfigure}[b]{0.27\textwidth}
    \centering
        \includegraphics[width=0.97\textwidth]{FullLogCircle.png}
        \caption{$x^2+y^2=2+5\cdot \log |x+1|$.}
    \end{subfigure}
    \hspace{1cm}
    \begin{subfigure}[b]{0.264\textwidth}
        \centering
        \includegraphics[width=\textwidth]{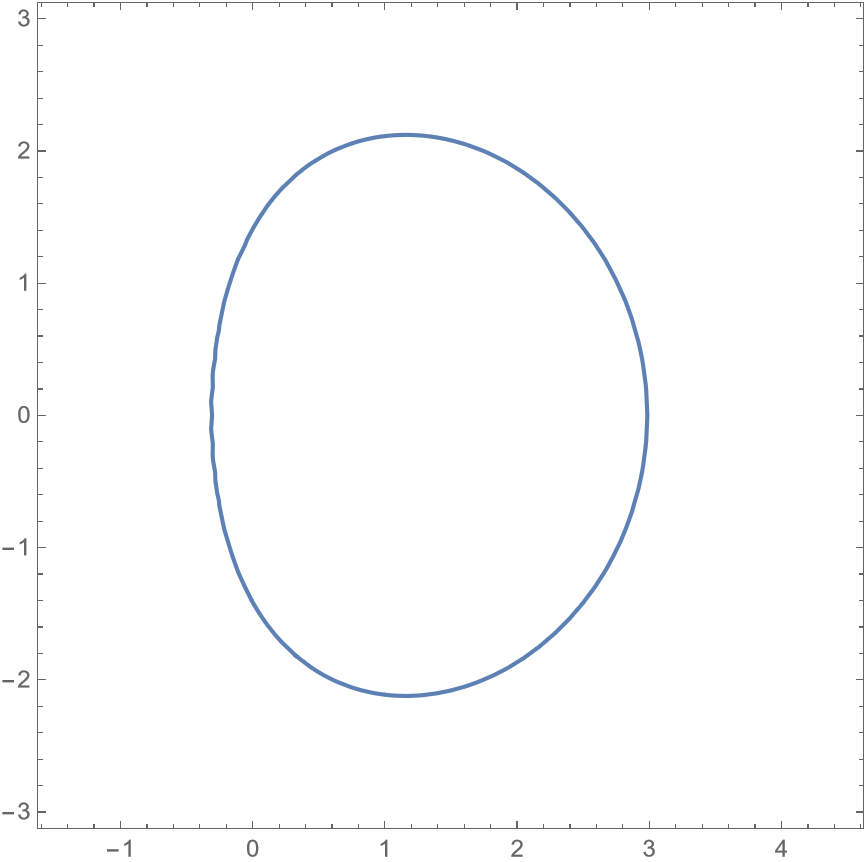}
        \caption{$x^2+y^2=20+\log |x+1|$.}
    \end{subfigure}
    \caption{Different Types of Log Circles.}
    \label{fig:LogCirclesCases}
\end{figure}

We now prove Theorem \ref{th:LogCirclesFewDistances}. We first recall the statement of this result.\vspace{2mm}

\noindent {\bf Theorem \ref{th:LogCirclesFewDistances}.}
\emph{Let $C_1$ and $C_2$ be matching log-circles. 
Then $(C_1,C_2)$ spans few distances.}

\begin{proof}
By definition, we may rotate, translate, and scale $\R^3$ so that $C_1$ and $C_2$ are defined as in \eqref{eq:MatchingLogCirclesDef}. 
Such transformations do not change the number of distances between two point sets. 
We first study the positive arcs of $C_1$ and $C_2$. 
These arcs are defined by
\begin{align*}
(x-B)^2+y^2&=D+A\cdot \ln(x) \quad \text{ and } \quad z=0, \\[2mm]
x^2+z^2&=D'+A\cdot \ln(x-B) \quad \text{ and } \quad y=0.
\end{align*}

We parameterize the two arcs as
\begin{align*}
\gamma_1(s)&=\left(s,\sqrt{D+A\cdot \ln (s)-(s-B)^2},0\right), \\[2mm]
\gamma_2(t)&=\left(t+B,0,\sqrt{D'+A\cdot \ln (t)-(t+B)^2}\right).
\end{align*}
The distance function of these parameterizations is 
\begin{align*}
\rho(s,t)
&=(s-t-B)^2+\left(D+A\cdot\ln (s) -(s-B)^2\right)+\left(D'+A\cdot \ln (t) -(t+B)^2\right)\\
&=\left(s^2+t^2+B^2-2st-2sB+2tB\right)+\left(D+A\cdot \ln (s) -s^2+2sB-B^2\right) \\
&\hspace{80mm} +\left(D'+A\cdot \ln (t) -t^2-2tB-B^2\right)\\[2mm]
&=-B^2-2st+D+D'+A\cdot \ln (s) +A\cdot \ln (t) =-B^2+D+D'-2st+A\cdot \ln (st).
\end{align*}

Since $\rho(s,t)$ is of a special form, Lemma \ref{le:SpecialFormFewDistances} implies that the two arcs span few distances.

The other cases are handled similarly. 
When considering the negative arc of $C_1$, we replace $\log(s)$ with $\log(-s)$.
When considering the negative arc of $C_2$, we replace $\log(t)$ with $\log(-t)$.
In all cases, The distance function $\rho(s,t)$ is of a special form. 
\end{proof}

\section{Few distances between curves on perpendicular planes} \label{sec:PerpendicularPlanes}

In this section, we prove Theorem \ref{th:PerpendicularPlanesFewDistances}.
More precisely, we prove a more general result about arcs in perpendicular planes.

\begin{theorem} \label{th:PerpendicularPlanesSpecialForm}
Let $A_1,A_2\subset \R^3$ be arcs that are contained in perpendicular planes, with respective parametrizations $\gamma_1,\gamma_2$. 
Let the distance function $\rho(s,t)$ of $\gamma_1$ and $\gamma_2$ be of a special form. Then there exist nonempty open subsets of $A_1$ and $A_2$ that are contained in one of the following pairs:
\begin{enumerate}[noitemsep,topsep=1pt,label=(\roman*)]
    \item a line $\ell$ and a curve contained in a plane orthogonal to $\ell$, \label{ca:LineOrthPlane}
    \item parallel lines, \label{ca:ParLines}
    \item a line $\ell$ and an ellipse contained in a cylinder centered around $\ell$, \label{ca:LineEllCyl}
    \item CPO parabolas, \label{ca:CongOppPar}
    \item matching curves, \label{ca:Matching}
    \item matching log-circles. \label {ca:LogCircles}
\end{enumerate}
\end{theorem}

Before proving Theorem \ref{th:PerpendicularPlanesSpecialForm}, we observe that Theorem \ref{th:PerpendicularPlanesFewDistances} is a corollary of Theorem  \ref{th:PerpendicularPlanesSpecialForm}.
We first recall the statement of that result.
\vspace{2mm}

\noindent {\bf Theorem \ref{th:PerpendicularPlanesFewDistances}.}
\emph{Let $C_1$ and $C_2$ be curves that are contained in perpendicular planes in $\R^3$. \\[2mm]
(a) Assume that $C_1$ and $C_2$ are one of the following: 
\begin{itemize}[noitemsep,topsep=0pt]
    \item a line $\ell$ and a curve contained in a  plane orthogonal to $\ell$, 
    \item parallel lines, 
    \item a line $\ell$ and an ellipse contained in cylinder with axis $\ell$, 
    \item CPO parabolas, 
    \item matching curves, 
    \item perpendicular circles.
\end{itemize}
Then $(C_1,C_2)$ spans few distances. \\[2mm]
(b) If $(C_1,C_2)$ is not one of the configuration from part (a), then $(C_1,C_2)$ spans many distances.}
\begin{proof}
(a) Immediate from Theorem \ref{th:LogCirclesFewDistances} and Lemmas \ref{le:LineCylinderFewDistances}, \ref{le:LineOrthogonalPlaneFewDistances}, \ref{le:PerpendicularParabolasFewDistances}, and \ref{le:MatchingEllipseHyperbolaFewDistances}.

(b) Consider a set of $m$ points $\pts_1\subset C_1$ and a set of $n$ points $\pts_2\subset C_2$. 
By Lemma \ref{le:SmoothParametrization} there exist arcs $A_1\subset C_1$ and $A_2\subset C_2$ such that $|\pts_1\cap A_1|=\Theta(m)$ and $|\pts_2\cap A_2|=\Theta(n)$.
Let $\gamma_1(s)$ and $\gamma_2(t)$ be the respective parameterizations of $A_1$ and $A_2$.
Assume for contradiction that $(A_1,A_2)$ does not span many distances. 
By Lemma \ref{le:SpecialFormOrExpands}, the distance function $\rho(s,t)$ of $\gamma_1(s)$ and $\gamma_2(t)$ is of a special form. 

By Theorem \ref{th:PerpendicularPlanesSpecialForm}, there exist nonempty open subsets $A_1'\subset A_1$ and $A_2'\subset A_2$ that are contained in a specific pair of sets. The six possible pairs of sets are numbered (i)--(vi) in the statement of Theorem \ref{th:PerpendicularPlanesSpecialForm}.
Assume that we are in one of cases (i)--(v), which describe pairs of irreducible curves.
Then each of $C_1$ and $C_2$ has an infinite intersection with the corresponding curve.
Theorem \ref{th:MilnorThom} implies that $C_1$ and $C_2$ share connected components with the respective curves. 
Since these are irreducible curves, we conclude that $C_1$ and $C_2$ are the two curves from Theorem \ref{th:PerpendicularPlanesSpecialForm}.
This is a contradiction, since the current theorem assumes that $C_1$ and $C_2$ are none of these pair of curves. 

It remains to consider case (vi) of Theorem \ref{th:PerpendicularPlanesSpecialForm}. 
That is, the case where $A_1$ and $A_2$ are contained in matching log-circles.
Since log-circles include a logarithm in their defining equations, we cannot use Theorem \ref{th:PerpendicularPlanesSpecialForm} as in the preceding paragraph.
However, a set that is defined as in \eqref{eq:LogCircle} still has a bounded intersection with a curve. 
For example, see Chapter 1 of Khovanski\u{\i} \cite{Khovanskii91}.
The rest of the argument remains as in the preceding paragraph and leads to a contradiction. 
\end{proof}

The rest of this section is dedicated to proving Theorem  \ref{th:PerpendicularPlanesSpecialForm}.
We first prove three lemmas that are required for the proof of that theorem. 
The reader might prefer to skim those lemmas and return to them later, if needed.

With Lemma \ref{le:LineAlgebraic} in mind, the following lemma studies ellipses that are contained in a cylinder.
The notation $\langle u,v \rangle$ denotes the standard dot product of $u,v \in \R^3$. 

\begin{lemma}\label{le:EllipseInCylinder}
For nonzero $\mu,D \in \R$, we consider the line 
\[ \ell = \Vb(x-\mu z, y)\subset \R^3, \]
and the ellipse
\[
C = \Vb\left(\frac{x^2}{\mu^2+1}+y^2-D, z\right) \subset \R^3.
\]
Then $C$ is contained in a cylinder centered around $\ell$.
\end{lemma}
\begin{proof}
We consider a unit vector in the direction of $\ell$:  \[ v=\left(\frac{\mu}{\sqrt{1+\mu^2}},0,\frac{1}{\sqrt{1+\mu^2}}\right).\] 
We also consider a point 
\[ p=\left(p_x,\sqrt{D-\frac{p_x^2}{\mu^2+1}},0\right)\in C. \]

The orthogonal projection of $p$ on $\ell$ is
\[ \langle p,v\rangle \cdot v = p_x\cdot \frac{\mu}{\sqrt{1+\mu^2}} \cdot v = \left(\frac{\mu^2p_x}{1+\mu^2},0,\frac{\mu p_x}{1+\mu^2}\right). \]
We set $q=\langle p,v\rangle \cdot v$.
The distance between $p$ and $\ell$ is the distance between $p$ and $q$. 
The square of this distance is
\begin{align*} \left(p_x-\frac{\mu^2p_x}{1+\mu^2}\right)^2 &+ D-\frac{p_x^2}{\mu^2+1} +  \left(\frac{\mu p_x}{1+\mu^2}\right)^2  \\[2mm]
&=D+p_x^2\cdot\frac{(1+\mu^2)^2-2\mu^2(1+\mu^2)+\mu^4-(1+\mu^2)+\mu^2}{(1+\mu^2)^2} = D.
\end{align*}
Since this holds for every $p\in C$, we conclude that $C$ is contained in a cylinder of radius $\sqrt{D}$ around $\ell$.
\end{proof}

The following lemma characterizes when a distance function $\rho(s,t)$ does not depend on $s$ or $t$. 
Recall that $\rho_s$ and $\rho_t$ are the first partial derivatives of $\rho(s,t)$ according to $s$ and $t$, respectively.

\begin{lemma}\label{le:DependsOnST}
Let $\gamma_1(s)$ and $\gamma_2(t)$ be parametrizations of arcs $A_1$ and $A_2$, respectively. 
Let $\rho(s,t)$ be the distance function of $\gamma_1(s)$ and $\gamma_2(t)$.
Assume that $\rho_s$ or $\rho_t$ is identically zero.
Then one of the two arcs is contained in a circle $C$ and the other arc is contained in the axis of $C$.
\end{lemma}
\begin{proof}
Without loss of generality, suppose that $\rho_s$ is identically zero. 
We consider two points $p_1,p_2\in A_2$.
By the assumption on $\rho_s$, the points of $A_1$ are equidistant from $p_1$.
That is, $A_1$ is contained in a sphere $S_1$ centered at $p_1$.
Similarly, $A_1$ is contained in a sphere $S_2$ centered at $p_2$.
Since $S_1 \cap S_2$ is a circle, we get that $A_1$ is contained in a circle $C$.
An arc of the circle $C$ is equidistant from a point $p$ if and only if $p$ is on the axis of $C$.
Thus, $A_2$ is contained in the axis of $C$.
\end{proof}

We also require the following highly specialized result. 

\begin{lemma}\label{le:ImportantLemma}
For nonempty open intervals $I_1,I_2\subset \R$, consider smooth functions $f:I_1\to \R$ and $g:I_2\to \R$, such that 
\begin{align*}
\sigma(s,t)&=(s-t)^2+f(s)^2+g(t)^2.
\end{align*}
If 
\begin{equation}\label{eq:TwoSides}
\frac{\partial^2(\ln|\sigma_s/\sigma_t|)}{\partial s \partial t}\equiv 0,
\end{equation}
then there exist $A_1,A_2,A_3\in \R$ that satisfy one of the following cases:
\begin{enumerate}[label=(\roman*)]
    \item $f(s)f'(s)=A_1+A_2s-s$, \quad $A_2\neq 0$, \quad and $\quad g(t)g'(t)=-A_1/A_2+t/A_2-t,$
    \item $f(s)f'(s)=\frac{A_1}{s+A_2}+A_3-s\quad $ and $\quad g(t)g'(t)=\frac{A_1}{t-A_3}-A_2-t.$
\end{enumerate}
\end{lemma}
\begin{proof}
We set $\mu(s)=s+f(s)f'(s)$ and $\nu(t)=t+g(t)g'(t)$.
Then 
 \begin{align}
 \sigma_s&=2(s-t)+2f(s)f'(s)=2(\mu(s)-t), \nonumber\\[2mm]
 \sigma_t&=2(t-s)+2g(t)g'(t)=2(\nu(t)-s). \label{eq:SigmaDerivs}
\end{align}
Let $\sigma_{st}$ be the second derivative of $\sigma$ according to $s$ and $t$ and note that $\sigma_{st}=-2$.

\parag{Studying $\mu(s)$ and $\nu(t)$.} By splitting the quotient inside the log at \eqref{eq:TwoSides}, we get that
 \begin{equation*} 
 \frac{\partial^2 \ln |\sigma_s|}{\partial s \partial t}
 =\frac{\partial^2 \ln |\sigma_t|}{\partial s \partial t}.
 \end{equation*}
 Applying the derivative rule for logarithms leads to 
 \[ \frac{\partial  (\sigma_{st}/\sigma_s)}{\partial s}
 =\frac{\partial (\sigma_{st}/\sigma_t)}{ \partial t}. \]
Combining this with the above values for $\sigma_s, \sigma_t, \sigma_{st}$ leads to
 \[
 \frac{\partial (-2/(2(\mu(s)-t)))}{\partial s}
=\frac{\partial(-2/(2(\nu(t)-s)))}{\partial t},
 \]
 
By computing the above derivatives, we get that 
 \begin{equation}\label{eq:EqualDerivs}
 \frac{\mu'(s)}{(\mu(s)-t)^2}=\frac{\nu'(t)}{(\nu(t)-s)^2}.
 \end{equation}
The assumption \eqref{eq:TwoSides} implies that $\sigma_s$ and $\sigma_t$ are never zero. 
Combining this with \eqref{eq:SigmaDerivs} leads to $\mu(s)\neq t$ and $\nu(t)\neq s$. 
Thus, the denominators in \eqref{eq:EqualDerivs} are never zero.

For a fixed $t\in I_2$, we set $C_1=t$, $C_2=\nu'(t)$, and $C_3=\nu(t)$.
Then,  \eqref{eq:EqualDerivs} becomes
\[
\frac{\mu'(s)}{(\mu(s)-C_1)^2}=\frac{C_2}{(C_3-s)^2}.
\]
Integrating both sides according to $s$ leads to 
\[
\frac{1}{\mu(s)-C_1}=\frac{-C_2}{C_3-s}+C_4 = \frac{C_3C_4-C_4s-C_2}{C_3-s},
\]
for some $C_4\in \R$.
Since the left side of this equation is nonzero, we have that $C_3C_4-C_4s-C_2\neq 0$.
We may thus rearrange this equation as 
\[
\mu(s)= \frac{C_3-s}{C_3C_4-C_4s-C_2}-C_4
\]

When $C_4=0$, we can rewrite the above as $\mu(s)=B_1+B_2s$, for some $B_1,B_2\in \R$. When $C_4\neq 0$, we can rewrite the above as 
\[
\mu(s)=\frac{B_1}{B_2+s}+B_3.
\]
We call the first form of $\mu(s)$ the \emph{linear form} and the second the \emph{rational form}.
A symmetric argument with a fixed $s$ implies that either $\nu(t)=B_1'+B_2't$ or $\nu(t)=\frac{B'_1}{B'_2+t}+B_3'$.

We rewrite \eqref{eq:EqualDerivs} as \begin{equation}\label{eq:CrossMultiplied}
\mu'(s)(\nu(t)-s)^2=\nu'(t)(\mu(s)-t)^2.
\end{equation}

\parag{The case of two linear forms.} We first assume that both $\mu(s)$ and $\nu(t)$ have a linear form. 
Plugging these forms into \eqref{eq:CrossMultiplied} gives
\begin{align*}
B_2(B_1'+B_2't-s)^2&=B_2'(B_1+B_2s-t)^2, \quad \text{ or equivalently} \\[2mm]
B_2(B_1'^2+B_2'^2t^2+s^2+2B_1'B_2't-2B_1's-2B_2'st)
&=B_2'(B_1^2+B_2^2s^2+t^2+2B_1B_2s-2B_1t-2B_2st).
\end{align*}
Cancelling the identical last term on both sides leads to
\begin{equation}\label{eq:Reduced}
B_2(B_1'^2+B_2'^2t^2+s^2+2B_1'B_2't-2B_1's)
=B_2'(B_1^2+B_2^2s^2+t^2+2B_1B_2s-2B_1t)
\end{equation}
By comparing the coefficients of $s^2$ on both sides, we obtain that $B_2=B_2^2B_2'$.
The coefficients of $t^2$ imply that $B_2'=B_2(B_2')^2$.
Combining these two observations implies either $B_2=B_2'=0$ or $B_2B_2' = 1$.
When $B_2=B_2'=0$, both $\mu(s)$ and $\nu(t)$ are constant functions, so we are in case (ii) of the lemma with $A_1=0$.
When $B_2B_2' = 1$, comparing the coefficients of $t$ in \eqref{eq:Reduced} implies that $B_1'=-B_1B_2'$.
We can rephrase this as $B_1'=-B_1/B_2$, which leads to $\nu(t)=-B_1/B_2+t/B_2$.
This is case (i).

\parag{The case of one linear form.}
We next consider the case where $\mu(s)$ has a linear form and $\nu(t)$ has a rational form.
In this case, \eqref{eq:CrossMultiplied} becomes  
\[
B_2\left(\frac{B_1'}{B_2'+t}+B_3'-s\right)^2
=\frac{-B_1'}{(B_2'+t)^2}(B_1+B_2s-t)^2.
\]
Multiplying both sides by $(B_2'+t)^2$ leads to
\[
B_2(B_1'+(B_3'-s)(B_2'+t))^2=-B_1'(B_1+B_2s-t)^2.
\]
Comparing the coefficients of $s^2t^2$ on both sides implies that $B_2=0$. 
This turns the above equation to $0=B_1'(B_1-t)^2$, which in turn implies that $B_1'=0$.
Once again, $\mu(s)$ and $\nu(t)$ are constants, so we are in case (ii) with $A_1=0$. 

The case where $\mu(s)$ has a rational form and $\nu(t)$ has a linear form is handled symmetrically.

\parag{The case of two rational forms.}
We consider the case where both $\mu(s)$ and $\nu(t)$ have a rational form. In this case,  \eqref{eq:CrossMultiplied} becomes
\[
\frac{B_1}{(B_2+s)^2}\left(\frac{B_1'}{B_2'+t}+B_3'-s\right)^2
=\frac{B_1'}{(B_2'+t)^2}\left(\frac{B_1}{B_2+s}+B_3-t\right)^2.
\]
Multiplying by $(B_2+s)^2(B_2'+t)^2$ leads to
\[
B_1(B_1'+(B_3'-s)(B_2'+t))^2=B_1'(B_1+(B_3-t)(B_2+s))^2.
\]
Comparing the coefficients of $s^2t^2$ implies that $B_1=B_1'.$ If $B_1=B_1'=0$ then $\mu(s)$ and $\nu(t)$ are constants and we are in case (ii) with $A_1=0$. Otherwise, the above equation becomes 
\begin{align*}
(B_1+(B_3'-s)(B_2'+t))^2&=(B_1+(B_3-t)(B_2+s))^2, \quad \text{ or equivalently} \\[2mm]
B_1+(B_3'-s)(B_2'+t)&=\pm(B_1+(B_3-t)(B_2+s)).
\end{align*}

The $\pm$ symbol must be replaced with a plus, since otherwise the coefficients of $st$ do not match. 
Comparing the coefficients of $s$ implies that $B_2'=-B_3$.
Comparing the coefficients of $t$ implies that $B_3'=-B_2$.
We conclude that $\nu(t)=\frac{B_1}{-B_3+t}-B_2$, which is case (ii).
\end{proof}

We are finally ready to prove Theorem \ref{th:PerpendicularPlanesSpecialForm}. 

\begin{proof}[Proof of Theorem \ref{th:PerpendicularPlanesSpecialForm}.]
By rotating and translating $\R^3$, we may assume that $A_1$ is contained in the $xy$-plane and $A_2$ is contained in the $xz$-plane. 
That is, both $\gamma_{1,z}(s)$ and $\gamma_{2,y}(t)$ are identically zero. 
If the $x$-coordinate of $\gamma_1$ is constant on an interval in $(0,1)$, then $A_1$ contains a line segment parallel to the $y$-axis.
This implies that we are in case \ref{ca:LineOrthPlane}.
A symmetric argument holds when $\gamma_2$ has constant $x$-coordinate in an interval of $(0,1)$. 
We may thus assume that $\gamma_{1,x}$ and $\gamma_{2,x}$ are not constant in any interval.

We next consider the case where $\rho_s\equiv 0$ or $\rho_t\equiv 0$ in a nonempty neighborhood of $(0,1)^2$. 
By defining $A_1'$ and $A'_2$ as the corresponding subarcs,  Lemma \ref{le:DependsOnST} implies that we are in case \ref{ca:LineOrthPlane}. We may thus assume that $\rho_s\not\equiv 0$ and $\rho_t \not\equiv 0$ in every neighborhood. 
By Lemma \ref{le:DerivativeTest}, there exist nonempty open intervals $I_1,I_2 \subset (0,1)$ such that 
\[
\frac{\partial^2 \ln|\rho_s/\rho_t|}{\partial s \partial t}\equiv 0 \quad \text{in}\ I_1\times I_2.
\]

Since $\gamma_{1,x}$ and $\gamma_{2,x}$ are not constant in any interval, there exist $s_0\in I_1$ and $t_0\in I_2$ such that $\gamma_{1,x}'(s_0)\neq 0$ and $\gamma_{2,x}'(t_0)\neq 0.$ 
By the inverse function theorem, there exist a neighborhood $J_1\subset \R$ of $\gamma_{1,x}(s_0)$ and a neighborhood $J_2\subset \R$ of $\gamma_{2,x}(t_0)$ in which $\gamma_{1,x}^{-1}$ and $\gamma_{2,x}^{-1}$ are well-defined and smooth. 
We consider the parameterizations
$\widetilde{\gamma}_1(s)=\gamma_1\circ \gamma_{1,x}^{-1}$ and
$\widetilde{\gamma}_2(t)=\gamma_2\circ \gamma_{2,x}^{-1}(t)$.
Note that $\widetilde{\gamma}_1(s)$ parameterizes a subarc of $A_1$, when restricted to $J_1$. 
Similarly, $\widetilde{\gamma}_2(t)$ parameterizes a subarc of $A_2$, when restricted to $J_2$.  
 
Let $\widetilde{\rho}(s,t)$ be the distance function of $\widetilde{\gamma}_1(s)$ and $\widetilde{\gamma}_2(t)$. 
This is a slight abuse of notation, since the domain of this distance function is $J_1\times J_2$ rather than $(0,1)^2$.
We also set $f(s)=\widetilde{\gamma}_{1,y}(s)$ and $g(t)=\widetilde{\gamma}_{2,z}(t)$.
Since $\widetilde{\gamma}_{1,x}(s)=\gamma_{1,x}((\gamma_{1,x}^{-1}(s))=s$ and
 $\widetilde{\gamma}_{2,x}(t)=\gamma_{2,x}(\gamma_{2,x}^{-1}(t))=t$, we have that
\begin{align*}
 \widetilde{\rho}(s,t)
 &=(\widetilde{\gamma}_{1,x}(s)-\widetilde{\gamma}_{1,x}(t))^2
 +\widetilde{\gamma}_{1,y}(s)^2
 +\widetilde{\gamma}_{2,z}(t)^2\\[2mm]
 &=(s-t)^2
 +f(s)^2+g(t)^2.
 \end{align*}
Since $\widetilde{\rho}(s,t)=\rho(\gamma_{1,x}^{-1}(s),\gamma_{2,x}^{-1}(t))$, we have that
 \begin{equation*}
 \frac{\partial^2\ln|\widetilde{\rho}_s/\widetilde{\rho}_t|}{\partial s \partial t}\equiv 0 \quad \text{in}\ J_1\times J_2. 
 \end{equation*}

To complete the proof, it suffices to show that there exist continuous open subsets of $\widetilde{\gamma}_1(J_1)$ and $\widetilde{\gamma}_2(J_2)$ that are contained in one of the pairs from the statement of the theorem.
By the above, we may apply Lemma \ref{le:ImportantLemma} with $\widetilde{\rho}(s,t)$. This lemma implies that there exist $A_1,A_2,A_3\in \R$ that satisfy one of the following cases:
\begin{enumerate}[label=(\alph*)]
    \item $f(s)f'(s)=A_1+A_2s-s$, \quad $A_2\neq 0$, \quad and $\quad g(t)g'(t)=-A_1/A_2+t/A_2-t,$
    \item $f(s)f'(s)=\frac{A_1}{s+A_2}+A_3-s\quad $ and $\quad g(t)g'(t)=\frac{A_1}{t-A_3}-A_2-t.$
\end{enumerate}

We note that 
\[
(f(s)^2)'=2f(s)f'(s) \qquad \text{ and } \qquad
(g(t)^2)' = 2g(t)g'(t).
\]
This in turn implies that 
\begin{equation} \label{eq:SquareIntegral}
f(s)^2=2\int f(s)f'(s)\, ds \qquad \text{ and } \qquad g(t)^2 = 2\int g(t)g'(t)\, dt.
\end{equation}

\parag{\textbf{Case (a).}}
We first assume that we are in case (a). 
Then there exist $A_4,A_5\in \R$ such that \eqref{eq:SquareIntegral} becomes 
\begin{align}
f(s)^2 &= 2A_1s+s^2(A_2-1) + A_4, \nonumber \\[2mm]
g(t)^2 &= -2A_1t/A_2+t^2(1/A_2-1)+A_5. \label{eq:LinearCaseAfterIntegration}
\end{align}

When $A_1=0$ and $A_2=1$, both $f(s)$ and $g(t)$ are constant functions. 
Then, $\widetilde{\gamma}_1(J_1)$ and $\widetilde{\gamma}_2(J_2)$ are line segments parallel to the $x$-axis.
We are thus in case \ref{ca:ParLines} of the theorem.

When $A_1\neq 0$ and $A_2=1$, we get that
\begin{align*}
\widetilde{\gamma}_1(J_1) &= \{(s,\sqrt{2A_1s+A_4},0):s\in J_1\}, \\[2mm]
\widetilde{\gamma}_2(J_2) &= \{(t,0,\sqrt{A_5-2A_1t}):t\in J_2\}.
\end{align*}
We note that $\widetilde{\gamma}_1(J_1)$ is a segment of the parabola $\Vb(z,x-(y^2-A_4)/2A_1)$.
Similarly, $\widetilde{\gamma}_2(J_2)$ is a segment of the parabola $\Vb(y,x+(z^2-A_5)/2A_1)$.
It is not difficult to verify that these are CPO parabolas. 
In particular, the axis of both parabolas is the $x$-axis of $\R^3$.
We are thus in Case \ref{ca:CongOppPar}. 

It remains to consider the case where $A_2 \neq 1$. 
In this case, there exist $A_4',A_5'\in \R$ such that we can rewrite \eqref{eq:LinearCaseAfterIntegration} as
\begin{align*}
f(s)^2 &= 2A_1s+s^2(A_2-1) + A_4 = (A_2-1)\left(s+\frac{A_1}{A_2-1}\right)^2+A'_4,  \\
g(t)^2 &= -2A_1t/A_2+t^2(1/A_2-1)+A_5 = \frac{1-A_2}{A_2}\left(t+\frac{A_1}{A_2-1}\right)^2+A_5'. \end{align*}

Imitating the above analysis  that led to the CPO parabolas, we obtain that
\begin{align*}
\widetilde{\gamma}_1(J_1) &\subset \Vb\left(z,y^2-(A_2-1)\left(x+\frac{A_1}{A_2-1}\right)^2-A_4'\right), \\
\widetilde{\gamma}_2(J_2) &\subset \Vb\left(y,z^2-\frac{1-A_2}{A_2}\left(x+\frac{A_1}{A_2-1}\right)^2-A_5'\right).
\end{align*}

We translate $\R^3$ a distance of $-A_1/(A_2-1)$ in the $x$-direction. 
We set $m=1-A_2$ and note that $m/(m-1) = (A_2-1)/A_2$.
By the assumption of this case, $A_2\neq 0$, so $m\neq 1$.
Then, the above becomes
\begin{align*}
\widetilde{\gamma}_1(J_1) &\subset \Vb\left(z,y^2+mx^2-A_4'\right), \nonumber \\
\widetilde{\gamma}_2(J_2) &\subset \Vb\left(y,z^2+\frac{m}{m-1}x^2-A_5'\right).
\end{align*}

We set $k=m/(m-1)$ and note that $m=k/(k-1)$.
This leads to a symmetry between  $\widetilde{\gamma}_1(J_1)$ and $\widetilde{\gamma}_2(J_2)$: To switch between the two curves, we switch $m$ and $k$, and the $y$ and $z$ axes.
If $m<0$ then $k>0$.
Thus, by possibly switching $\widetilde{\gamma}_1(J_1)$ and $\widetilde{\gamma}_2(J_2)$, we may assume that $m>0$.
This implies that $A_4'> 0$, since otherwise $\widetilde{\gamma}_1(J_1)$ consists of at most one point.
When $A_5'\neq 0$, the above is a matching pair, so we are in Case \ref{ca:Matching}. 

It remains to consider the case where $A_5' = 0$. 
We have that $0<m<1$, since otherwise $\widetilde{\gamma}_2(J_2)$ consists of at most one point.
We rewrite the above as
\[ x = \pm z \cdot \sqrt{\frac{1-m}{m}}.\]
We conclude that  $\widetilde{\gamma}_2(J_2)$ is contained in the union of two lines.
By the definition of $\widetilde{\gamma}_2(J_2)$, it is a segment of one of these two lines $\ell$.

Setting the slope of $\ell$ as $\mu = \pm \sqrt{\frac{1-m}{m}}$, we note that $1/(\mu^2+1)=m$.
We may thus apply Lemma \ref{le:EllipseInCylinder} with $\ell$ and the ellipse that contains $\widetilde{\gamma}_1(J_1)$. The lemma states that the ellipse is contained in a cylinder centered around $\ell$.
This is Case \ref{ca:LineEllCyl}.

\parag{\textbf{Case (b).}}
We now assume that we are in case (b). 
Then there exist $A_4,A_5\in \R$ such that \eqref{eq:SquareIntegral} becomes 
\begin{align*}
f(s)^2 &= 2A_1\ln |s+A_2|+2A_3s-s^2+A_4,  \\[2mm]
g(t)^2 &= 2A_1 \ln |t-A_3|-2A_2t-t^2+A_5.
\end{align*}
Repeating the analysis of Case (a) leads to\footnote{We abuse the notation $\Vb(\cdot)$ by placing a function that is not a polynomial in it. }
\begin{align*}
\widetilde{\gamma}_1(J_1) &\subset \Vb\left(z,x^2+y^2-2A_1\ln |x+A_2|-2A_3x-A_4\right), \nonumber \\
\widetilde{\gamma}_2(J_2) &\subset \Vb\left(y,x^2+z^2-2A_1 \ln |x-A_3|+2A_2x-A_5\right).
\end{align*}

After rearranging, we get that there exist $A_4',A_5' \subset \R$, such that
\begin{align*}
\widetilde{\gamma}_1(J_1) &\subset \Vb\left(z,(x-A_3)^2+y^2-2A_1\ln |x+A_2|-A'_4\right), \nonumber \\
\widetilde{\gamma}_2(J_2) &\subset \Vb\left(y,(x+A_2)^2+z^2-2A_1 \ln |x-A_3|-A'_5\right).
\end{align*}

The above sets are matching log-circles. Indeed, consider the sets after translating $\R^3$ a distance of $A_2$ in the $x$-direction.
This is Case \ref{ca:LogCircles}.
\end{proof}

\section{Few distances between conics}\label{sec:FewDistancesBetweenConics}

In this section, we prove Theorem \ref{th:FewDistancesBetweenConics}. 
This theorem is a corollary of the results of Section \ref{sec:Configurations} and of following lemma.

\begin{lemma} \label{le:EightConicCases}
In the following, all curves are in $\R^3$. \\[2mm]
(a) Two circles that are neither aligned nor perpendicular span many distances. \\[2mm]
(b) A circle and hyperbola span many distances. \\[2mm]
(c) Two parabolas that are not CPO span many distances.\\[2mm]
(d) A parabola and an ellipse span many distances. \\[2mm]
(e) A parabola and a hyperbola span many distances. \\[2mm]
(f) Two hyperbolas span many distances. \\[2mm]
(g) A hyperbola and an ellipse that do not match span many distances. \\[2mm]
(h) Two ellipses that do not match and are not aligned circles span many distances.
\end{lemma}
\begin{proof}
Part (a) is proved in \cite{MS20}, so it remains to consider cases (b)--(h). 
Since the different proofs are almost identical, we present them together. 
We denote the two curves as $C_1$ and $C_2$.
We first perform a rotation, translation, and uniform scaling of $\R^3$ to obtain the following.
\begin{itemize}[noitemsep,topsep=1pt]
    \item If $C_1$ is a parabola then $C_1 = \vb(y-x^2,z)$.
    \item If $C_1$ is a hyperbola then $C_1 = \vb(x^2-\frac{y^2}{c^2}-1,z)$.
    \item If $C_1$ is an ellipse then $C_1 = \vb(x^2+\frac{y^2}{c^2}-1,z)$.
\end{itemize}
Such transformations do not change the number of distances between two sets of points. 

After the above transformations, we have a nice parameterization $\gamma_1(s)$ of $C_1$, possibly excluding a single point: 
\begin{itemize}[noitemsep,topsep=1pt]
    \item If $C_1$ is a parabola then $\gamma_1(s) = (s,s^2,0)$.
    \item If $C_1$ is a hyperbola then $\gamma_1(s) = \left(\frac{s^2+1}{2s},c\cdot\frac{1-s^2}{2s},0\right)$.
    \item If $C_1$ is an ellipse then $\gamma_1(s) = \left(\frac{1-s^2}{1+s^2},c\cdot\frac{2s}{1+s^2},0\right)$.
\end{itemize}

After fixing $C_1$, we cannot transform $\R^3$ to also give $C_2$ a simple parameterization. 
To parameterize $C_2$, we use a point $(p,q,r)\in \R^3$ and orthogonal vectors $\mathbf{v}_1,\mathbf{v}_2\in \R^3$: 
\begin{itemize}[noitemsep,topsep=1pt]
    \item If $C_2$ is a parabola then $\gamma_2(t) = (p,q,r)+t\cdot  \mathbf{v}_1+t^2\cdot  \mathbf{v}_2$.
    \item If $C_2$ is a hyperbola then $\gamma_2(t) = (p,q,r)+\frac{t^2+1}{2t}\cdot \mathbf{v}_1+\frac{t^2-1}{2t}\cdot \mathbf{v}_2$.
    \item If $C_2$ is an ellipse then $\gamma_2(t) = (p,q,r)+\frac{t^2-1}{t^2+1}\mathbf{v}_1+\frac{2t}{t^2+1} \mathbf{v}_2$.
\end{itemize}
As before, these parameterizations might exclude a single point of $C_2$.

Let $\rho(s,t)$ be the distance function of $\gamma_1(s)$ and $\gamma_2(t)$, as defined in \eqref{eq:DistanceFunc}.
By Lemma \ref{le:SpecialFormOrExpands}, either $\rho(s,t)$ is of a special form or $(C_1,C_2)$ spans many distances. 
We assume that we are in the case where $\rho(s,t)$ is of a special form.
If $\rho_s$ or $\rho_t$ is identically zero in an open neighborhood, then Lemma \ref{le:DependsOnST} implies that $C_1$ or $C_2$ is a line. 
Since this contradicts the assumption of the theorem, $\rho_s$ and $\rho_t$ are not identically zero in any open neighborhood.

By the above, we may apply Lemma \ref{le:DerivativeTest} with $\rho(s,t)$. 
The lemma implies that there exists an open neighborhood in $\R^2$ in which 
\begin{equation} \label{eq:SpecialFormAgain}
\frac{\partial^2(\ln|\rho_s/\rho_t|)}{\partial s \partial t}\equiv 0.
\end{equation}
A careful reader might complain that Lemma \ref{le:DerivativeTest} is defined for pararmeterizations with the domain $(0,1)$,
while the above parameterizations assume $s,t\in \R$.
It is easy to address this  technicality.
For example, we can obtain half of the curve by composing the above parameterization with $h(r)=1/r$, where $r\in (0,1)$.
The other half is obtained by composing the parameterization with $h(r)=-1/r$.
By observing  \eqref{eq:SpecialForm}, we note that having a special form is invariant under replacing $s$ with $\pm 1/s$ and $t$ with $\pm 1/t$. 

The left side of \eqref{eq:SpecialFormAgain} is a rational function. 
To see that, we set $h(s, t) = \rho_s/\rho_t$. 
Since $\rho(s, t)$ is a rational function in $s$ and $t$, so is $h(s, t)$.
For each point $(s, t)$ that satisifies $\rho_s(s,t)\neq 0$ and $\rho_t(s,t)\neq 0$, there exists an open neighbourhood where $|h(s, t)| = c \cdot h(s, t)$, with $c\in \{-1,1\}$. In either case, 
\[ \frac{\partial}{\partial t} \log |h(s, t)| = \frac{1}{c\cdot h(s, t)} \cdot \frac{\partial}{\partial t} (c \cdot h(s, t)) = \frac{h_t(s, t)}{h(s, t)}. \]

We write $\mathbf{v}_1 = (x_1,y_1,z_1)$ and $\mathbf{v}_2 = (x_2,y_2,z_2)$.
Then the left side of  \eqref{eq:SpecialFormAgain} is a rational function in $s$ and $t$ whose coefficients are rational functions in $c,p,q,r,x_1,x_2,y_1,y_2,z_1,z_2$. 
We denote the numerator of this rational function as $\rho_{N}(s,t)$.

By \eqref{eq:SpecialFormAgain}, the polynomial $\rho_{N}(s,t)$ is identically zero in an open neighborhood of $\R^2$. 
That is, all the coefficients of $\rho_{N}(s,t)$ are zero. 
We use Wolfram Mathematica to compute $\rho_{N}(s,t)$.
We then show that, when we are not in one of the special configurations that span few distances, it is impossible for all the coefficients of $\rho_{N}(s,t)$ to be zero simultaneously.
This contradiction implies that $(C_1,C_2)$ spans many distances, as asserted. 
For the Mathematica computations, see Appendices \ref{app:Wolfram} and \ref{app:Wolfram2}.
\end{proof}

We are now ready to prove Theorem \ref{th:FewDistancesBetweenConics}. 
We first recall the statement of this result.
\vspace{2mm}

\noindent {\bf Theorem \ref{th:FewDistancesBetweenConics}} (Few distances between conics)
\emph{Let $C_1,C_2\subset \R^3$ be non-degenerate conic sections. \\[2mm]
(a) Assume that $C_1$ and $C_2$ are either perpendicular circles, aligned circles, CPO parabolas, or matching curves. Then $(C_1,C_2)$ spans few distances. \\[2mm]
(b) If $(C_1,C_2)$ is not one of the configurations from part (a), then $(C_1,C_2)$ spans many distances.}

\begin{proof}
Part (a) is immediate by combining Lemmas \ref{le:ParallelPlanes}, \ref{le:AlignedCirclesFewDistances}, \ref{le:PerpendicularParabolasFewDistances}, and \ref{le:MatchingEllipseHyperbolaFewDistances}.
Part (b) is immediate from Lemma \ref{le:EightConicCases}.
\end{proof}

\appendix

\section{Mathematica computations: the case of two ellipses} \label{app:Wolfram}

In this appendix, we complete the proof of Lemma \ref{le:EightConicCases} for ellipses, by using Wolfram Mathematica. 
In particular, we assume that all coefficients of $\rho_N(s,t)$ are zero, and show that this can only happen when the ellipses match. 
The case of two ellipses is the most complicated one, by far. This case includes all ideas that are used in the other cases. 
In Appendix \ref{app:Wolfram2}, we see that the case of two hyperbolas is orders of magnitude simpler.

\parag{Final preparations.}
Before getting to the code, we further simplify the situation. 
Every ellipse centered at the origin in $\R^2$ can be defined as $\Vb(ax^2+bxy+cy^2-1)$ for some $a,b,c\in\R$. 
Conversely, $\Vb(ax^2+bxy+cy^2-1)$ is an ellipse when $a,c>0$ and $b^2-ac<0$.
An ellipse is \emph{standard} if it is centered at the origin and its axes are also  the axes of $\R^2$. 
Equivalently, an ellipse is standard if $b=0$ in its defining equation.

\begin{lemma} \label{le:OneVectorFromAxis}
(a) Consider an ellipse $E\subset \R^2$ centered at the origin. 
Let $p$ be the point with the smallest $y$-coordinate on $E$.
Then there exist $\mathbf{v}_1,\mathbf{v}_2\in \R^2$ such that $\mathbf{v}_1$ is on the $x$-axis and 
\begin{equation*} E\setminus \{p\}= \left\{\frac{1-s^2}{1+s^2}\cdot \mathbf{v}_1+\frac{2s}{1+s^2}\cdot \mathbf{v}_2\ :\ s\in \R \right\}.
\end{equation*}
(b) Let $\Pi\subset\R^3$ be a plane that contains the origin and let $E\subset \R^3$ be an ellipse centered at the origin. 
Then there exist $p\in E$, $\mathbf{v}_1\in \Pi$, and $\mathbf{v}_2\in \R^3$ such that 
\begin{equation*} E\setminus \{p\}= \left\{\frac{1-s^2}{1+s^2}\cdot \mathbf{v}_1+\frac{2s}{1+s^2}\cdot \mathbf{v}_2\ :\ s\in \R \right\}.
\end{equation*}
\end{lemma}
\begin{proof}
(a) We write $E = \Vb(ax^2+bxy+cy^2-1)$ and consider the transformation $T(x,y) = (x+by/2a,y)$.
A point $(x,y)$ is on $T(E)$ if and only if 
\begin{align*}
1
&=a\left(x-\frac{b}{2a}y\right)^2+b\left(x-\frac{b}{2a}y\right)y+cy^2\\
&=a\left(x^2-\frac{b}{a}xy+\frac{b^2}{4a^2}y^2\right)+bxy-\frac{b^2}{2a}y^2+cy^2=ax^2+\left(-\frac{b^2}{4a}+c\right) y^2.
\end{align*}

Since $b^2-ac<0$, the coefficient of $y^2$ is positive, so $T(E)$ is a standard ellipse.
We set $E' = T(E)$ and let $p'$ be the point with the smallest $y$-coordinate on $E'$.
Since $E'$ is standard, there exist axis-parallel vectors $\mathbf{v}'_1,\mathbf{v}'_2\in \R^2$ such that 
\[ E'\setminus \{p'\}= \left\{\frac{1-s^2}{1+s^2}\cdot \mathbf{v}'_1+\frac{2s}{1+s^2}\cdot \mathbf{v}'_2\ :\ s\in \R \right\}. \]

We note that $T(\cdot)$ does not change $y$-coordinates. 
Since $E' = T(E)$, we obtain the assertion of the lemma by setting $\mathbf{v}_1=T^{-1}(\mathbf{v}_1')$ and $\mathbf{v}_2=T^{-1}(\mathbf{v}_2')$. 

(b) Let $\Pi'$ be the plane that contains $C$. 
We may assume that $\Pi \neq \Pi'$, since otherwise we are done.
Since both planes contain the origin, their intersection is a line $\ell$.
We move to work inside of $\Pi'$ and change the axes so that $\ell$ is the $x$-axis. 
Then, part (a) of the current lemma completes the proof of part (b).
\end{proof}

We follow the parameterizations from the proof of Lemma \ref{le:EightConicCases} in Section \ref{sec:FewDistancesBetweenConics}.
However, instead of asking the vectors $\vb_1,\vb_2$ to be orthogonal, we apply Lemma \ref{le:OneVectorFromAxis}(b) to have the third coordinate of $\vb_2$ be 0. (That is, we apply Lemma \ref{le:OneVectorFromAxis}(b) after a translation taking the center of $C_2$ to be the origin.) 
Then, there exist $c>0$ and vectors $\bv_1 = (x_1,y_1,z_1),\bv_2 = (x_2,y_2,0)$ such that 
\begin{align}
\gamma_1(s) &= \left(\frac{1-s^2}{1+s^2},c\cdot\frac{2s}{1+s^2},0\right), \nonumber \\
\gamma_2(t) &= (p,q,r)+\frac{t^2-1}{t^2+1}\mathbf{v}_1+\frac{2t}{t^2+1} \mathbf{v}_2. \label{eq:AppendixParam}
\end{align}

If both $C_1$ and $C_2$ are circles, then we are done by Theorem \ref{th:ShefferMathialagan}.
Thus, without loss of generality, we may assume that $C_1$ is not a circle. 
In other words, $c\neq 1$.
By Lemma \ref{le:ParallelPlanes}, we may assume that $C_2$ is not contained in a plane parallel to the $xy$-plane.
This in turn implies that $z_1\neq 0$.

Let $\rho(s,t)$ be the distance function of $\gamma_1(s)$ and $\gamma_2(t)$, as defined in \eqref{eq:DistanceFunc}.
As in Section \ref{sec:FewDistancesBetweenConics}, we set $\rho_{N}(s,t)$ to be the numerator of $\frac{\partial^2(\ln|\rho_s/\rho_t|)}{\partial s \partial t}$.
We recall that $\rho_{N}(s,t)$ is a polynomial in $s$ and $t$, where the coefficients of the monomials depend on $c,x_1,y_1,z_1,x_2,y_2,p,q,r$.
To complete the proof, we assume that all these coefficients equal zero and show that this only happens when $C_1$ and $C_2$ are matching ellipses. 
We denote the coefficient of $s^at^b$ in $\rho_{N}(s,t)$ as $\coeff[s^at^b]$.

\parag{The case analysis.}
We are now ready to present the computational part of our analysis, 
which is partitioned into cases.
We start with the most general case and gradually branch into more specialized cases.
We include figures with the code of the first four Mathematica calculations.
Since the code is very similar for all cases, we believe that four figures suffice.
For readability, in the code we denote $\gamma_1$ and $\gamma_2$ as $c1$ and $c2$, respectively. 
The function \emph{dist[s,t]} represents $\rho(s,t)$. 
The variable \emph{RhoTest} represents $\rho_{N}(s,t)$.

The case where $p=q=r=0$ is addressed separately below, under the heading ``The case of two ellipses centered at the origin.'' 
We first assume that $C_2$ is not centered at the origin.
Since matching ellipses have the same center, we cannot have matching ellipses in this case.
Figure \ref{fi:RoadmapA} may be helpful while reading the case analysis.

\begin{figure}[htp]     \centering
    \includegraphics[width=1.1\textwidth]{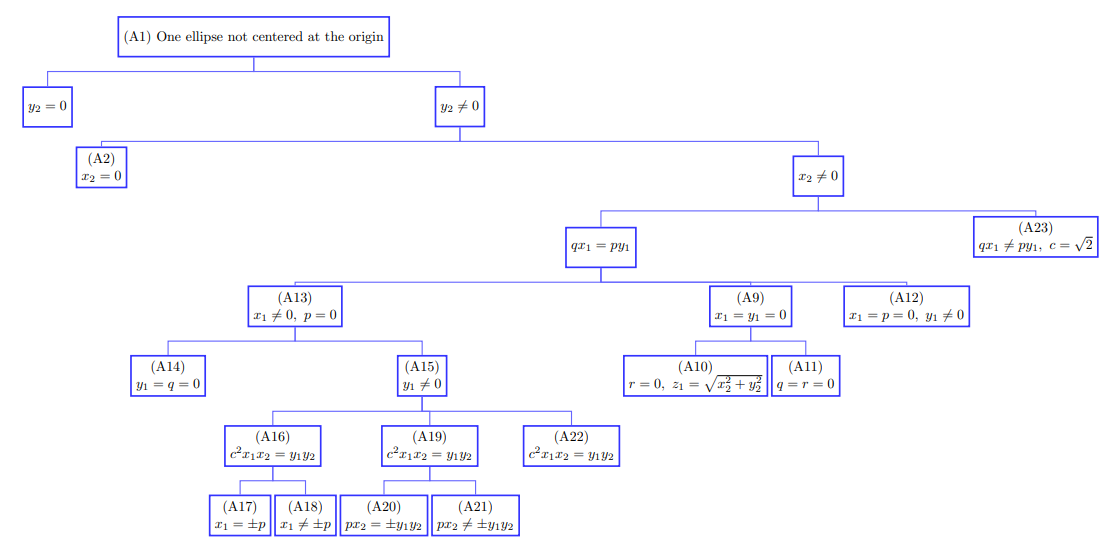}
    \caption{A roadmap of the case analysis when the second ellipse is not centered at the origin.}\label{fi:RoadmapA}
\end{figure}

\begin{enumerate}[label=(A\arabic*)] 
\item \label{ca:general}{\bf The general case.} For our first step, we compute 
\[ \coeff[st]-\coeff[s^9t] = -192 c^2 (-2+c^2)x_2(qx_1-py_1)y_2. \] 
See Figure \ref{fi:generalCase}.
Since each coefficient of $\rho_{N}(s,t)$ equals zero, so does the difference of two coefficients.
Since $c>0$, we have that $x_2=0$, or $y_2=0$, or $c=\sqrt{2}$, or $qx_1=py_1$. 

 \begin{figure}[h]
    \centering
    \includegraphics[width=0.8\textwidth]{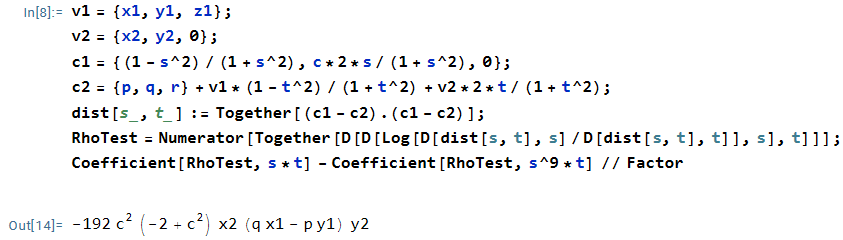}
   \caption{The Mathematica code of Case \ref{ca:general}. }
\label{fi:generalCase}

\end{figure}

If $y_2=0$ then we swap the $x$ and $y$ axes, to again reach the case where $x_2=0$. 
In this case $y_2\neq 0$, since  $v_2\neq (0,0,0)$.
Case \ref{ca:x2=0y2neq0} assumes that $x_2=0$ and $y_2\neq 0$. 
Case \ref{ca:x1x2y2neq0y1q=0} assumes that $x_2,y_2\neq 0$, $qx_1=py_1$, and $x_1\neq 0$. 
Case \ref{ca:x1y1=0x2y2neq0} assumes that $x_2,y_2\neq 0$, $qx_1=py_1$, and $x_1=y_1=0$.
Case \ref{ca:x1=p=0y1x2y2neq0} assumes that $x_2,y_2\neq 0$, $x_1=0$, $y_1\neq 0$, and $p=0$.
Case \ref{ca:c=sqrt2} assumes that $x_2, y_2\neq 0$, $qx_1 \neq py_1$, and $c=\sqrt{2}$. 
These cases cover all possible options.
 
\item \label{ca:x2=0y2neq0} {\bf\boldmath The case where $x_2=0$ and $y_2\neq 0$.} As shown in Figure \ref{fi:x2=0y2neq0}, we compute \[ \coeff[s^3t]-\coeff[s^5t]/2=384(-1+c)c^2(1+c)x_1y_{2}^2. \] 
Since each coefficient is equal to 0, so does this difference.
Since $c>0$ and $c\neq 1$, we conclude that $x_1=0$.

 \begin{figure}[h]
    \centering
    \includegraphics[width=0.8\textwidth]{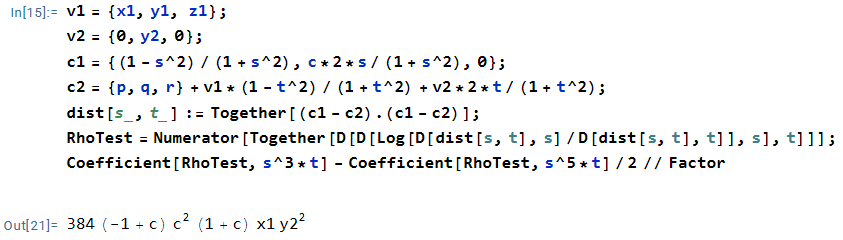}
   \caption{The Mathematica code of the first part of Case \ref{ca:x2=0y2neq0}. }
\label{fi:x2=0y2neq0}
\end{figure}

Next, as shown in Figure \ref{fi:x2=0y2neq0part2}, we compute 
\[ \coeff[s^3t]+\coeff[s^3t^9] =256c^2py_{2}^2(3qy_1+rz_1).\] 
Since $c,y_2\neq 0$, either $p=0$ or $3qy_1+rz_1=0$. 
Case \ref{ca:x1=x2=p=0yneq0} assumes that $p=0$.
Case \ref{x1=x2=0y1pneq0r=-3qy1/z1} assumes that $p\neq 0$ and $3qy_1+rz_1=0$. 
Since $z_1\neq 0$, we may rewrite $3qy_1+rz_1=0$ as $r=-3qy_1/z_1$.

 \begin{figure}[h]
    \centering
    \includegraphics[width=0.8\textwidth]{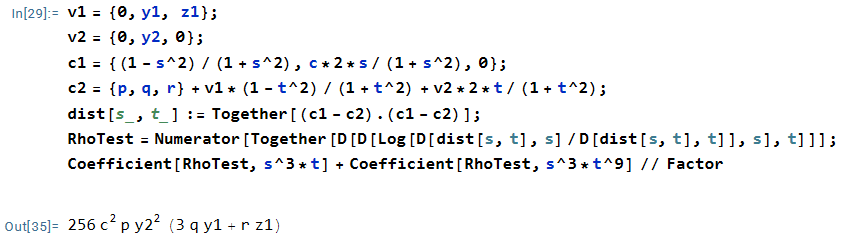}
   \caption{The Mathematica code of the second  part of Case \ref{ca:x2=0y2neq0}. }
\label{fi:x2=0y2neq0part2}
\end{figure}

\item \label{ca:x1=x2=p=0yneq0} {\bf\boldmath The case where $x_1=x_2=p=0$ and $y_2\neq 0$.}
As shown in Figure \ref{fi:x1=x2=p=0yneq0}, we compute
\begin{align*} 10(\coeff[s]-\coeff[st^{10}])+\coeff[st^6]+\coeff[s^5t^4] &=-384(-1+c)c^2(1+c)qy_2z_{1}^2, \\
\coeff[s]+\coeff[st^{10}]&=-32(-1+c)c^2(1+c)y_2z_1(qr+y_1z_1).
\end{align*}
Since $c> 0$, $c\neq 1$, and $y_2,z_1\neq 0$, the first expression above implies that $q=0$. 
The second expression leads to $qr+y_1z_1=0$. 
Since $q=0$ and $z_1\neq 0$, we conclude that $y_1=0$.

 \begin{figure}[h]
    \centering
    \includegraphics[width=\textwidth]{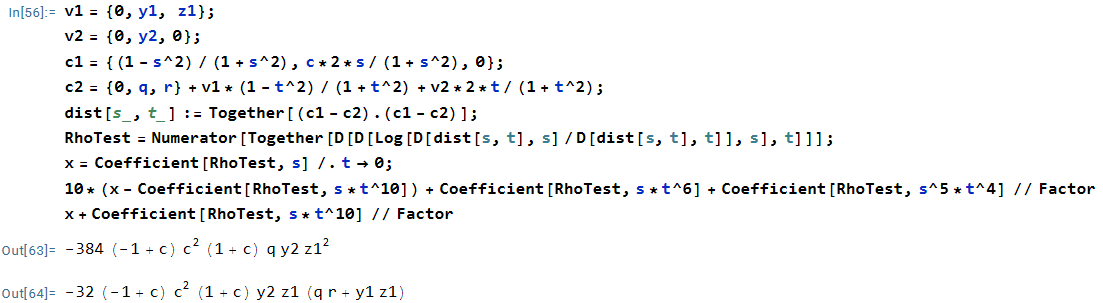}
   \caption{The Mathematica code of Case \ref{ca:x1=x2=p=0yneq0}. }
\label{fi:x1=x2=p=0yneq0}
\end{figure}

Next, we compute $\coeff[st^3]=-8(-1+c)c^2(1+c)ry_{2}^2z_1$. 
Since $c>0$, $c>1$, and $y_2\neq 0$, we get that $r=0$.
This contradicts the assumption that $C_2$ is not centered at the origin, so the current case cannot occur.

\ignore{ \begin{figure}[h]
    \centering
    \includegraphics[width=\textwidth]{MathematicaScreenshots/Snip5.PNG}
    \label{}
\end{figure}    }

\ignore{\begin{figure}[h]
    \centering
    \includegraphics[width=\textwidth]{MathematicaScreenshots/Snip6.PNG}
    \label{}
\end{figure}}

\item \label{x1=x2=0y1pneq0r=-3qy1/z1} {\bf\boldmath The case where $x_1=x_2=0$, $y_2,p\neq 0$, and $r=-3qy_1/z_1$.}
In this case, we compute $\coeff[s^3t^5]+\coeff[s^7t^5]=-4096c^2pqy_{1}^3$. 
Since $c,p\neq 0$, either $q=0$ or $y_1=0$.
Case \ref{ca:x1=x2=q=r=0y2pneq0} assumes that $q=0$. 
Case \ref{ca:x1=x2=y1=r=0y2pqneq0} assumes that $q\neq 0$ and $y_1=0$. 
In both cases, the assumption $r=-3qy_1/z_1$ becomes $r=0$.

\ignore{ \begin{figure}[h]
    \centering
    \includegraphics[width=\textwidth]{MathematicaScreenshots/Snip7.PNG}
    \label{}
\end{figure}  }

\item \label{ca:x1=x2=q=r=0y2pneq0} {\bf\boldmath The case where $x_1=x_2=q=r=0$ and $y_2,p\neq 0$.}
In this case, we compute $\coeff[s^5]=-32c^2py_1y_2(2y_2-z_1)(2y_2+z_1)$. 
Since $c,p,y_2\neq 0$, we get that $y_1=0$ or $z_1=\pm 2y_2$.
Case \ref{ca:x1=x2=y1=q=r=0y2pneq0} assumes that $y_1=0$.
Case \ref{x1=x2=q=r=0y1y2pneq0-z1=2y2} assumes that $y_1\neq 0$ and $z_1=2y_2$.
If $z_1=-2y_2$, then we perform a reflection of $\R^3$ about the $xz$-plane. 
This leads to $z_1=2y_2$ and does not affect anything else.
Thus, Case \ref{x1=x2=q=r=0y1y2pneq0-z1=2y2} also covers $z_1=-2y_2$.

\ignore{
 \begin{figure}[h]
    \centering
    \includegraphics[width=\textwidth]{MathematicaScreenshots/Snip8.PNG}
    \label{}
\end{figure} }

\item \label{ca:x1=x2=y1=q=r=0y2pneq0} {\bf\boldmath The case where $x_1=x_2=y_1=q=r=0$ and $y_2,p\neq 0$.}
In this case, we compute $\coeff[s^3t]=-8c^2py_{2}^2(y_2-z_1)(y_2+z_1)$. 
Since $c,p,y_2\neq 0$, we get that $y_2=\pm z_1$. 
Without loss of generality, we assume that $y_2=z_1$.
If $y_2=-z_1$, we reflect $\R^3$ about the $xz$-plane. 

\ignore{
 \begin{figure}[h]
    \centering
    \includegraphics[width=\textwidth]{MathematicaScreenshots/Snip9.PNG}
    \label{}
\end{figure} }

We next compute $3\coeff[1]+\coeff[s^2]=6cpz_1$.
Since $c,p,z_1\neq 0$,
 this coefficient is nonzero. 
This contradiction implies that the current case cannot occur.

\ignore{
 \begin{figure}[h]
    \centering
    \includegraphics[width=\textwidth]{MathematicaScreenshots/Snip10.PNG}
    \label{}
\end{figure}}

\item \label{x1=x2=q=r=0y1y2pneq0-z1=2y2} {\bf\boldmath The case where $x_1=x_2=q=r=0$, $y_1,y_2,p\neq 0$, and $z_1=2y_2$.} 
We compute $\coeff[s^3t
]=384c^2py_{2}^2(y_{1}^2+y_{2}^2)$. 
Since $c,p,y_1,y_2\neq 0$, this coefficient cannot be zero.
This contradiction implies that the current case cannot occur.

\ignore{
 \begin{figure}[h]
    \centering
    \includegraphics[width=\textwidth]{MathematicaScreenshots/Snip11.PNG}
    \label{}
\end{figure} }

\item \label{ca:x1=x2=y1=r=0y2pqneq0} {\bf\boldmath The case where $x_1=x_2=y_1=r=0$ and $y_2,p,q\neq 0$.}
In this case, we compute  $\coeff[s^5t^5
]=256c^2py_{2}^2(y_2-z_1)(y_2+z_1)$.
Since this coefficient equals zero and $c,p,y_2\neq 0$, we have that $y_2=\pm z_1$. 
As before, it suffices to consider the case of $y_2=z_1$.
The other case is identical after a reflection across the $xz$-plane. 

\ignore{
 \begin{figure}[h]
    \centering
    \includegraphics[width=\textwidth]{MathematicaScreenshots/Snip12.PNG}
    \label{}
\end{figure} }

We next compute $3\coeff[1]+\coeff[s^2]=24cpz_1$. Since $c,p,z_1\neq 0$, the above expression is nonzero.
This contradiction implies that the current case cannot occur.

\ignore{
 \begin{figure}[h]
    \centering
    \includegraphics[width=\textwidth]{MathematicaScreenshots/Snip13.PNG}
    \label{}
\end{figure} }

\item \label{ca:x1y1=0x2y2neq0}{\bf\boldmath The case where $x_1=y_1=0$ and $x_2,y_2\neq 0$.}
We compute \begin{align*}
\coeff[st^5]-\coeff[s^9t^5]&=-512(-1+c)c^2(1+c)ry_{2}^2z_1, \\
\coeff[st^9]&=-96c^2qx_2y_2(x_{2}^2+y_{2}^2+rz_1-z_{1}^2).
\end{align*}
Since $c>0$, $c\neq 1$, and $y_2,z_1\neq 0$, the first computation above leads to $r=0$.
This leads to $\coeff[st^9]=-96c^2qx_2y_2(x_{2}^2+y_{2}^2-z_{1}^2)$.
Since this equals zero, $z_1^2 = x_2^2+y_2^2$ or $q=0$. 
We rewrite the former as $z_1 = \pm \sqrt{x_2^2+y_2^2}$.
As usual, we may assume that $z_1 = \sqrt{x_2^2+y_2^2}$, since the other option is symmetric after a reflection about the $xy$-plane. 
Case \ref{ca:x1=y1=r=0x2y2neq0z1Euclidean} assumes that $z_1 = \sqrt{x_2^2+y_2^2}$.
Case \ref{ca:x1=y1=r=q=0x2y2neq0} assumes that $q=0$.

 \ignore{
 \begin{figure}[h]
    \centering
    \includegraphics[width=\textwidth]{MathematicaScreenshots/Snip14.PNG}
    \label{}
\end{figure} }

\item \label{ca:x1=y1=r=0x2y2neq0z1Euclidean} {\bf\boldmath The case where $x_1=y_1=r=0$, $x_2,y_2\neq 0$, and $z_1 = \sqrt{x_{2}^2+y_{2}^2}$.}
We compute $\coeff[s^3]=32(-1+c)(1+c)x_2$. 
Since $c>0$, $c\neq 1$, and  $x_2\neq 0$, this coefficient is nonzero.  
This contradiction implies that the current case cannot occur.

\ignore{
 \begin{figure}[h]
    \centering
    \includegraphics[width=\textwidth]{MathematicaScreenshots/Snip15.PNG}
    \label{}
\end{figure} }

\item \label{ca:x1=y1=r=q=0x2y2neq0} {\bf\boldmath The case where $x_1=y_1=r=q=0$ and $x_2,y_2\neq 0$.}
We compute
\begin{align*}
\coeff[st^6]-\coeff[s^9]&=-8c^2(-1+c^2-p^2)x_2y_{2}^2, \\
\coeff[st^6]+\coeff[s^9]&=8c^2(-2+c^2)px_2y_{2}^2.
\end{align*}
Since $c,x_2,y_2\neq 0$, the first computation leads to $p^2=c^2-1$. 
Since $c\neq \pm 1$, we get that $p\neq 0$. 
Since the second computation equals zero, we obtain that $c=\sqrt{2}$. 
Plugging this into $p^2=c^2-1$ gives $p=\pm 1$. 
If $p=-1$ then we reflect $\R^3$ about the $yz$-plane. 
We may thus assume that $p=1$, without loss of generality. 

\ignore{
 \begin{figure}[h]
    \centering
    \includegraphics[width=\textwidth]{MathematicaScreenshots/Snip16.PNG}
    \label{}
\end{figure}
}

We next compute $\coeff[s^5]=16x_2y_{2}^2$.
Since $x_2, y_2\neq 0$, this coefficient is nonzero.
This contradiction implies that the current case cannot occur.

\ignore{
 \begin{figure}[h]
    \centering
    \includegraphics[width=\textwidth]{MathematicaScreenshots/Snip17.PNG}
    \label{}
\end{figure} }

\item \label{ca:x1=p=0y1x2y2neq0} {\bf\boldmath The case where $x_1=p=0$ and $y_1,x_2,y_2\neq 0$.} 
We compute
\[ \coeff[st]-\coeff[st^9]+(\coeff[st^3]-\coeff[st^7])/2=-384c^2rx_2y_1y_2z_1. \] 
Since $c,y_1,z_1,x_2,y_2\neq 0$, we get that $r=0$.
Since we assume that $C_2$ is not centered at the origin, this implies that $q\neq 0$.

\ignore{
 \begin{figure}[h]
    \centering
    \includegraphics[width=\textwidth]{MathematicaScreenshots/Snip18.PNG}
    \label{}
\end{figure} }

We next compute 
\[ \coeff[st^9]-\coeff[s^9t]=-192c^2qx_2y_2(x_{2}^2-3y_{1}^2+y_{2}^2-z_{1}^2).\]
Since the above equals zero, we have that $z_{1}=\pm \sqrt{x_{2}^2-3y_{1}^2+y_{2}^2}$. 
After possibly reflecting $\R^3$ about the $xy$-plane, we may assume that $z_1=\sqrt{x_{2}^2-3y_{1}^2+y_{2}^2}$.
We then compute  $\coeff[s^3t^9]-\coeff[s^7t]=256qx_2 y_{1}^2y_2$. 
Since $q, x_2, y_1, y_2\neq 0$, this expression is nonzero.
This contradiction implies that the current case cannot occur.

\ignore{
 \begin{figure}[h]
    \centering
    \includegraphics[width=\textwidth]{MathematicaScreenshots/Snip19.PNG}
    \label{}
\end{figure}}

\ignore{
 \begin{figure}[h]
    \centering
    \includegraphics[width=\textwidth]{MathematicaScreenshots/Snip23.PNG}
    \label{}
\end{figure} }

\item \label{ca:x1x2y2neq0y1q=0} {\bf\boldmath The case where $x_1,x_2,y_2\neq 0$ and $qx_1=py_1$.} 
Since $x_1\neq 0$, we may rewrite $qx_1=py_1$ as $q=py_1/x_1$.
We then compute  
\[ \coeff[st]-\coeff[st^9]=-192 c^2 x_{1}^3 x_2 y_1 (p x_1 x_2 y_1 + 2 p x_{1}^2 y_2 - p x_{2}^2 y_2 + 3 p y_{1}^2 y_2 - p y_{2}^3 + r x_1 y_2 z_1 + p y_2 z_{1}^2). \] 
For this to equal zero, either $y_1=0$ or $r = p (-x_1x_2y_1 - 2 x_{1}^2 y_2 + x_{2}^2 y_2 - 3 y_{1}^2 y_2 + y_{2}^3 - y_2z_{1}^2)/(x_1y_2z_1)$. When $y_1=0$, the assumption $q=py_1/x_1$ implies that $q=0$.
Case \ref{y1=q=0x1x2y2neq0} assumes that $y_1=q=0$.
Case \ref{co:x1x2y1y2neq0,q=py1/x1,rVeryLong} assumes that $y_1\neq 0$ and 
$r = p (-x_1x_2y_1 - 2 x_{1}^2 y_2 + x_{2}^2 y_2 - 3 y_{1}^2 y_2 + y_{2}^3 - y_2z_{1}^2)/(x_1y_2z_1)$.

\ignore{
 \begin{figure}[h]
    \centering
    \includegraphics[width=\textwidth]{MathematicaScreenshots/Snip24.PNG}
    \label{}
\end{figure} }

\item \label{y1=q=0x1x2y2neq0} {\bf\boldmath The case where $x_1,x_2,y_2\neq 0$ and $y_1=q=0$.}
We compute 
\begin{align*} 
\coeff[s] + \coeff[st^{10}] + \coeff[s^9] + \coeff[s^9t^{10}]&=-32 c^2 p x_1 x_2 y_{2}^2,\\
\coeff[s] + \coeff[st^{10}] - \coeff[s^9] -\coeff[s^9t^{10}]&=-16 c^2 (-2 + c^2) x_1 x_2 y_{2}^2.
 \end{align*}
Since $x_1,x_2,y_2,c\neq 0$, the first expression leads to $p=0$ and the second to $c=\sqrt{2}$. 

 \ignore{
 \begin{figure}[h]
    \centering
    \includegraphics[width=\textwidth]{MathematicaScreenshots/Snip25.PNG}
    \label{}
\end{figure} }

We next compute  $\coeff[st^2
]+\coeff[s^9t^8]=-96 r x_2 y_{2}^2 z_1$. 
Since this equals zero, we get that $r=0$.
This is impossible, since we assume that $C_2$ is not centered at the origin. 
Thus, this case cannot occur.

\item \label{co:x1x2y1y2neq0,q=py1/x1,rVeryLong} {\bf\boldmath The case where $x_1,x_2,y_1,y_2\neq 0$, $q=py_1/x_1$, and $r = p (-x_1x_2y_1 - 2 x_{1}^2 y_2 + x_{2}^2 y_2 - 3 y_{1}^2 y_2 + y_{2}^3 - y_2z_{1}^2)/(x_1y_2z_1)$.} 
We compute
\[ \coeff[s^3t]-\coeff[s^3t^9]=-256 p x_{1}^3 y_2 (-x_2 y_1 + x_1 y_2) (c^2 x_1 x_2 - y_1 y_2) (x_1 x_2 + y_1 y_2). \] 
For this expression to be zero, either $p=0$, or $x_2y_1=x_1y_2$, or $c^2x_1x_2=y_1y_2$, or $x_1x_2=-y_1y_2$.
When $p=0$, the assumptions imply that $q=r=0$.
This is impossible, since we assume that $C_2$ is not centered at the origin.
Case \ref{ca:x1x2y1y2neq0,q=py1/x1,c^2x1x2=y1y2,rVeryLong} assumes that $c^2x_1x_2=y_1y_2$.
Case \ref{ca:x1x2y1y2neq0,q=py1/x1,x1x2=-y1y2,rVeryLong} assumes that $x_1x_2=-y_1y_2$.
Case \ref{ca:x1x2y1y2pneq0,q=py1/x1,x2y1=x1y2,rVeryLong} assumes that 
$p\neq 0$ and $x_2y_1=x_1y_2$.

\ignore{
 \begin{figure}[h]
    \centering
    \includegraphics[width=\textwidth]{MathematicaScreenshots/Snip28.PNG}
    \label{}
\end{figure}}

\item \label{ca:x1x2y1y2neq0,q=py1/x1,c^2x1x2=y1y2,rVeryLong} {\bf\boldmath The case where $x_1,x_2,y_1,y_2\neq 0$, $q=py_1/x_1$, $c^2x_1x_2=y_1y_2$, and $r = p (-x_1x_2y_1 - 2 x_{1}^2 y_2 + x_{2}^2 y_2 - 3 y_{1}^2 y_2 + y_{2}^3 - y_2z_{1}^2)/(x_1y_2z_1)$.} 
We rewrite $c^2 x_1x_2 = y_1y_2$ as $y_1=c^2 x_1x_2/y_2$. 
We then compute
\begin{align*}
\hspace{-5mm}\coeff[st]=48 c^4 x_1 x_{2}^2 y_{2}^4 (&c^2 p^2 x_{1}^2 x_{2}^2 + 3 c^4 p^2 x_{1}^2 x_{2}^2 - c^2 x_{1}^4 x_{2}^2 - 3 c^4 x_{1}^4 x_{2}^2 + 2 x_{1}^2 y_{2}^2- 2 c^2 x_{1}^2 y_{2}^2 + 2 p^2 x_{1}^2 y_{2}^2 \\
& - 2 x_{1}^4 y_{2}^2 - 2 p^2 x_{2}^2 y_{2}^2 + 2 x_{1}^2 x_{2}^2 y_{2}^2 - 2 p^2 y_{2}^4 + 2 x_{1}^2 y_{2}^4 + 2 p^2 y_{2}^2 z_{1}^2 - 2 x_{1}^2 y_{2}^2 z_{1}^2). 
\end{align*}
Since $c,x_1,x_2,y_2\neq 0$, the expression in the parentheses equals zero. 
We rewrite this expression as  
\begin{align}
       2y_{2}^2z_{1}^2 (x_1^2 - p^2) =c^2 p^2 x_{1}^2 x_{2}^2 &+ 3 c^4 p^2 x_{1}^2 x_{2}^2 - c^2 x_{1}^4 x_{2}^2 - 3 c^4 x_{1}^4 x_{2}^2 + 2 x_{1}^2 y_{2}^2- 2 c^2 x_{1}^2 y_{2}^2 \nonumber \\
& + 2 p^2 x_{1}^2 y_{2}^2 - 2 x_{1}^4 y_{2}^2 - 2 p^2 x_{2}^2 y_{2}^2 + 2 x_{1}^2 x_{2}^2 y_{2}^2 - 2 p^2 y_{2}^4 + 2 x_{1}^2 y_{2}^4. \label{eq:ComplicatedCase}
   \end{align}
   
We denote the right-hand side of the above equation as $G$.   
Case \ref{ca:x1x2y1y2neq0,p=x1,q=y1=c^2x1x2/y2} assumes that $x_1=p$.
In that case, $G=0$ and the assumption $q=py_1/x_1$ implies that $q=y_1$.
We do not require the assumption on $r$ in case \ref{ca:x1x2y1y2neq0,p=x1,q=y1=c^2x1x2/y2}, so we ignore it.
If $x_1=-p$ then we negate $\vb_1$, to again obtain that $x_1=p$.
This does not change anything else in the current situation, so $x_1=-p$ is also covered by Case \ref{ca:x1x2y1y2neq0,p=x1,q=y1=c^2x1x2/y2}. 

When $x_1\neq \pm p$, Equation \eqref{eq:ComplicatedCase} leads to $z_1 = \pm \sqrt{G/(2y_{2}^2 (x_1^2 - p^2))}$. 
By potentially reflecting $\R^3$ about the $xy$-plane, we may  assume that $z_1 = \sqrt{G/(2y_{2}^2 (x_1^2 - p^2))}$.
This scenario is studied in Case \ref{ca:x1x2y1y2neq0,p=x1,q=py1/x1,y1=c^2x1x2/y2,x1neqpmp}.

\ignore{
 \begin{figure}[h]
    \centering
    \includegraphics[width=\textwidth]{MathematicaScreenshots/Snip29.PNG}
    \label{}
\end{figure} }

\item \label{ca:x1x2y1y2neq0,p=x1,q=y1=c^2x1x2/y2} {\bf\boldmath The case where $x_1,x_2,y_1,y_2\neq 0$, $p=x_1$, and $q= y_1=c^2x_1x_2/y_2$.} 
We compute $\coeff[t^{10}]=-(-1 + c) c (1 + c) x_{2}^2 y_{2}^7$.
Since $c>0$, $c\neq 1$ and $x_2,y_2\neq 0$, this coefficient is nonzero.
This contradiction implies that the current case cannot occur.

\ignore{
 \begin{figure}[h]
    \centering
    \includegraphics[width=\textwidth]{MathematicaScreenshots/Snip30.PNG}
    \label{}
\end{figure} }

\item \label{ca:x1x2y1y2neq0,p=x1,q=py1/x1,y1=c^2x1x2/y2,x1neqpmp} {\bf\boldmath The case where $x_1,x_2,y_1,y_2\neq 0$, $q=py_1/x_1$, $y_1=c^2x_1x_2/y_2$, $x_1 \neq \pm p$, $z_1 = \sqrt{G/(2y_{2}^2 (x_1^2 - p^2))}$, and $r = p (-x_1x_2y_1 - 2 x_{1}^2 y_2 + x_{2}^2 y_2 - 3 y_{1}^2 y_2 + y_{2}^3 - y_2z_{1}^2)/(x_1y_2z_1)$.} 
We compute $\coeff[s^3t]=320 (-1 + c) c^4 (1 + c) (p - x_1)^2 x_1 (p + x_1)^2 x_{2}^2 y_{2}^6$. 
Since $c>0$, $c\neq 1$, $x_1 \neq \pm p$, and $x_1,x_2, y_2\neq 0$, this coefficient is nonzero.
This contradiction implies that the current case cannot occur.

\ignore{ \begin{figure}[h]
    \centering
    \includegraphics[width=\textwidth]{MathematicaScreenshots/Snip31.PNG}
    \label{}
\end{figure} }

\item \label{ca:x1x2y1y2neq0,q=py1/x1,x1x2=-y1y2,rVeryLong} {\bf\boldmath The case where $x_1,x_2,y_1,y_2\neq 0$, $q=py_1/x_1$, $x_1x_2=-y_1y_2$, and $r = p (-x_1x_2y_1 - 2 x_{1}^2 y_2 + x_{2}^2 y_2 - 3 y_{1}^2 y_2 + y_{2}^3 - y_2z_{1}^2)/(x_1y_2z_1)$.} 
Since $x_2\neq 0$, we may rewrite $x_1x_2=-y_1y_2$ as $x_1=-y_1y_2/x_2$. 
We compute 
\begin{align*}
\coeff[st]=-96 c^2 x_{2}^5 y_1 y_{2}^3 (p^2 x_{2}^6 &- p^2 x_{2}^4 y_{1}^2 + p^2 x_{2}^4 y_{2}^2 - 
   x_{2}^2 y_{1}^2 y_{2}^2 + c^2 x_{2}^2 y_{1}^2 y_{2}^2 - p^2 x_{2}^2 y_{1}^2 y_{2}^2 \\
   &- 
   x_{2}^4 y_{1}^2 y_{2}^2 + x_{2}^2 y_{1}^4 y_{2}^2 - x_{2}^2 y_{1}^2 y_{2}^4 + y_{1}^4 y_{2}^4 - 
   p^2 x_{2}^4 z_{1}^2 + x_{2}^2 y_{1}^2 y_{2}^2 z_{1}^2).
   \end{align*}
Since $c,x_2,y_1,y_2\neq 0$, the expression in the parentheses equals zero. 
We rewrite this as  
 \begin{align*}
     z_{1}^2x_{2}^2 (p^2x_2^2 - y_{1}^2 y_{2}^2) = p^2 x_{2}^6 - p^2 x_{2}^4 y_{1}^2 &+ p^2 x_{2}^4 y_{2}^2 - 
   x_{2}^2 y_{1}^2 y_{2}^2 + c^2 x_{2}^2 y_{1}^2 y_{2}^2\\
   &- p^2 x_{2}^2 y_{1}^2 y_{2}^2 - 
   x_{2}^4 y_{1}^2 y_{2}^2 + x_{2}^2 y_{1}^4 y_{2}^2 - x_{2}^2 y_{1}^2 y_{2}^4 + y_{1}^4 y_{2}^4.
 \end{align*}

We denote the right-hand side of the above equation as $H$. 
Case \ref{ca:x1x2y1y1neq0,q=py1/x1,x1=-y1y2/x2,px2=y1y1} assumes that $px_2=y_1y_2$, so $H=0$.
We do not require the assumption on $r$ in case \ref{ca:x1x2y1y1neq0,q=py1/x1,x1=-y1y2/x2,px2=y1y1}, so we ignore it.
The case where $px_2=-y_1y_2$ reduces to case \ref{ca:x1x2y1y1neq0,q=py1/x1,x1=-y1y2/x2,px2=y1y1} after negating $\bv_1$. 
Case \ref{ca:x1x2y1y1neq0,q=py1/x1,x1=-y1y2/x2,px2neqy1y1,LongRandZ1} assumes that $px_2\neq \pm y_1y_2$.
This implies that 
$z_1=\sqrt{H/(x_{2}^2 (p^2 x_{2}^2 - y_{1}^2 y_{2}^2))}$, possibly after a reflection of $\R^3$ about the $xy$-plane.

\ignore{
 \begin{figure}[h]
    \centering
    \includegraphics[width=\textwidth]{MathematicaScreenshots/Snip32.PNG}
    \label{}
\end{figure} }

\item \label{ca:x1x2y1y1neq0,q=py1/x1,x1=-y1y2/x2,px2=y1y1} {\bf\boldmath The case where $x_1,x_2,y_1,y_2\neq 0$, $q=py_1/x_1$, $x_1=-y_1y_2/x_2$, and $px_2=y_1y_2$.}
Since $x_2\neq 0$, we may rewrite $px_2\neq \pm y_1y_2$ as $p=y_1y_2/x_2$. 
We compute  
\[ \coeff[st]=-24 (-1 + c) c^2 (1 + c) x_{2}^7 y_1 y_2. \] 

Since $c>0$, $c\neq 1$, and $x_2,y_1,y_2\neq 0$, this coefficient is nonzero. 
This contradiction implies that the current case cannot occur.

\ignore{
 \begin{figure}[h]
    \centering
    \includegraphics[width=\textwidth]{MathematicaScreenshots/Snip33.PNG}
    \label{}
\end{figure} }

\item \label{ca:x1x2y1y1neq0,q=py1/x1,x1=-y1y2/x2,px2neqy1y1,LongRandZ1} {\bf\boldmath The case where $x_1,x_2,y_1,y_2\neq 0$, $q=py_1/x_1$, $x_1=-y_1y_2/x_2$, $r = p (-x_1x_2y_1 - 2 x_{1}^2 y_2 + x_{2}^2 y_2 - 3 y_{1}^2 y_2 + y_{2}^3 - y_2z_{1}^2)/(x_1y_2z_1)$, $z_1=\sqrt{H/(x_{2}^2 (p^2 x_{2}^2 - y_{1}^2 y_{2}^2))}$, and $px_2\neq \pm y_1y_2$.} 
We compute 
\[ \coeff[s^3t]=64 (-1 + c) (1 + c) (1 + 6 c^2) x_{2}^5 y_1 y_{2}^3 (p x_2 - y_1 y_2)^2 (p x_2 + y_1 y_2)^2. \]
Since $c>0$, $c\neq 1$, $x_2,y_1,y_2\neq 0$, and  $px_2\neq \pm y_1y_2$, this coefficient is nonzero.
This contradiction implies that the current case cannot occur.

\ignore{
 \begin{figure}[h]
    \centering
    \includegraphics[width=\textwidth]{MathematicaScreenshots/Snip34.PNG}
    \label{}
\end{figure} }

\item \label{ca:x1x2y1y2pneq0,q=py1/x1,x2y1=x1y2,rVeryLong} {\bf\boldmath The case where $x_1,x_2,y_1,y_2,p\neq 0$, $q=py_1/x_1$, $x_2y_1 = x_1y_2$, and $r = p (-x_1x_2y_1 - 2 x_{1}^2 y_2 + x_{2}^2 y_2 - 3 y_{1}^2 y_2 + y_{2}^3 - y_2z_{1}^2)/(x_1y_2z_1)$.} 
Since $y_2\neq 0$, we may rewrite $x_2y_1 = x_1y_2$ as $x_1=x_2y_1/y_2$. 
We compute 
\[ \coeff[t]+\coeff[s^{10}t^9]=32 (-1 + c) c (1 + c) p x_{2}^4 y_1 y_{2}^8 (x_{2}^2 + y_{2}^2 - z_{1}^2). \]

Since $c>0$, $c\neq 1$, and $p,x_2,y_1,y_2\neq 0$, we conclude that  $z_1=\pm \sqrt{x_{2}^2 + y_{2}^2}$. 
We may assume that $z_1= \sqrt{x_{2}^2 + y_{2}^2}$, possibly after a reflection of $\R^3$ about the $xy$-plane. 

\ignore{
 \begin{figure}[h]
    \centering
    \includegraphics[width=\textwidth]{MathematicaScreenshots/Snip35.PNG}
    \label{}
\end{figure} }

We next compute 
\[ \coeff[t^5]+\coeff[s^{10}t^5]=-512 (-1 + c) c (1 + c) p x_2^4 y_1^3 y_2^4 (x_2^2 + y_2^2). \]

Since $c>0$, $c\neq 1$, $p,x_2,y_1,y_2\neq 0$, the above coefficient is nonzero.
This contradiction implies that the current case cannot occur.

\ignore{
 \begin{figure}[h]
    \centering
    \includegraphics[width=\textwidth]{MathematicaScreenshots/Snip36.PNG}
    \label{}
\end{figure} }

\item \label{ca:c=sqrt2} {\bf\boldmath The case where $x_2,y_2\neq 0$, $qx_1 \neq py_1$, and $c=\sqrt{2}$.} We compute
\[ \coeff[s^3t]-\coeff[s^7t]=2304 x_2 (q x_1 - p y_1) y_2. \]

By the assumptions of this case, the above expression is nonzero. 
This contradiction implies that the current case cannot occur.

\ignore{
 \begin{figure}[h]
    \centering
    \includegraphics[width=\textwidth]{MathematicaScreenshots/Snip37.PNG}
\end{figure} }

\end{enumerate}

\parag{The case of two  ellipses centered at the origin.}
It remains to consider the case where $p=q=r=0$.
That is, the case where both ellipses are centered at the origin. 
In the case, we switch to a different parameterization of $C_2$.
By Lemma~\ref{le:OneVectorFromAxis}(b), there exist $\bv_1 = (x_1,y_1,z_1),\bv_2 = (x_2,0,z_2)$ such that 
\begin{align}
\gamma_1(s) &= \left(\frac{1-s^2}{1+s^2},c\cdot\frac{2s}{1+s^2},0\right), \nonumber \\[2mm] 
\gamma_2(t) &= \frac{t^2-1}{t^2+1}\mathbf{v}_1+\frac{2t}{t^2+1} \mathbf{v}_2. \label{eq:AppendixParam2}
\end{align}

\begin{enumerate}[label=(B\arabic*)]
    
\item \label{caB:general} {\bf The general case.}
We compute $\coeff[st]=-48 c^2 x_2 y_{1}^2 (x_1 x_2 + z_1 z_2)$. 
Since this coefficient equals zero, we have that $x_2=0$, or $y_1=0$, or $x_1 x_2 = -z_1 z_2$.
Case \ref{caB:x2=0} assumes that $x_2=0$.
Case \ref{caB:y1=0} assumes that $y_1=0$.
Case \ref{caB:x1x2+z1z2=0} assumes that $x_2,y_1\neq0$ and $x_1x_2=-z_1z_2$.

\ignore{
\begin{figure}[h]
\centering
\includegraphics[width=\textwidth]{MathematicaScreenshots/SnipS1.PNG}
\label{}
\end{figure} }
 
\item \label{caB:x2=0} {\bf\boldmath The case where $x_2=0$.} 
Since $\bv_2\neq 0$, we have that $z_2\neq 0$.
We compute  
\[ \coeff[t]=-8 (-1 + c) c (1 + c) y_1 z_{1}^2 z_{2}^2.\] 
For this coefficient to equal zero, either $y_1=0$ or $z_1=0$.
Case \ref{caB:x2=y1=0} assumes that $y_1=0$.
Case \ref{caB:z1=x2=0} assumes that $y_1\neq0$ and $z_1=0$.

\ignore{
\begin{figure}[h]
\centering
\includegraphics[width=\textwidth]{MathematicaScreenshots/SnipS2.PNG}
\label{}
\end{figure} }

\item \label{caB:x2=y1=0} {\bf\boldmath The case where $y_1=x_2=0$ and $z_2\neq0$.} 
Since $\bv_1$ and $\bv_2$ span a plane, we have that $x_1\neq 0$.
We compute $\coeff[st]=16 (-1 + c) (1 + c) x_1 z_1^2 z_2^2$. 
Since $x_1,z_2\neq 0$ and $c\neq \pm 1$, we get that $z_1=0$.

\ignore{
\begin{figure}[h]
\centering
\includegraphics[width=\textwidth]{MathematicaScreenshots/SnipS3.PNG}
\label{}
\end{figure} }

We next compute  
\[ 2\coeff[st]+\coeff[st^3]=-64 (-1 + c) (1 + c) x_1 (c^2 x_{1}^2 + z_{2}^2 - c^2 z_{2}^2). \]
For this to equal zero, we must have that $c^2 x_{1}^2 + z_{2}^2 - c^2 z_{2}^2=0$. 
We rewrite this as $z_2^2 = c^2x_1^2/(c^2-1)$.

We claim that, in this case, $C_1$ and $C_2$ are matching ellipses. 
Indeed, with \eqref{eq:AppendixParam2} in mind, we obtain  \eqref{FirstConic} with $m=a=c^2$.
By the assumptions of this case, we get that $\gamma_2(t) = \frac{t^2-1}{t^2+1}(x_1,0,0)+\frac{2t}{t^2+1} (0,0,z_2)$.
Combining this with $z_2^2 = c^2x_1^2/(c^2-1)$ leads to  \eqref{SecondConic} with $m=c^2$ and $b=z_2^2$. 

\ignore{ 
\begin{figure}[h]
\centering
\includegraphics[width=\textwidth]{MathematicaScreenshots/SnipS4.PNG}
\label{}
\end{figure} }

\item \label{caB:z1=x2=0} {\bf\boldmath The case where $z_1=x_2=0$ and $y_1\neq0$.} 
We compute  
\[ 2\coeff[st]+\coeff[st^3]=-128 (-1 + c) c^2 (1 + c) x_1 y_{1}^2.\]
Since this equals zero, we have that $x_1=0$.
We then compute
\[ 2\coeff[s^2t]+\coeff[s^2t^3]=-128 (-1 + c) c (1 + c) y_1 (y_{1}^2 - z_{2}^2 + c^2 z_{2}^2). \]
Since this equals zero, we get that $y_{1}^2 - z_{2}^2 + c^2 z_{2}^2=0$.
We rewrite this as $z_2^2 = -y_1^2/(c^2-1)$.

We claim that, in this case, $C_1$ and $C_2$ are matching ellipses. 
Indeed, we switch the $x$ and $y$ axes and set $m=1/c^2$.
Then, with \eqref{eq:AppendixParam2} in mind, we obtain  \eqref{FirstConic} with $a=1$ (after switching the axes).
By the assumptions of this case, we get that $\gamma_2(t) = \frac{t^2-1}{t^2+1}(0,y_1,0)+\frac{2t}{t^2+1} (0,0,z_2)$.
Combining this with $z_2^2 = -y_1^2/(c^2-1)$ leads to  \eqref{SecondConic} with $b=y_2^2$ (after switching the axes). 

\ignore{
\begin{figure}[h]
\centering
\includegraphics[width=\textwidth]{MathematicaScreenshots/SnipS5.PNG}
\label{}
\end{figure} }

\ignore{
\begin{figure}[h]
\centering
\includegraphics[width=\textwidth]{MathematicaScreenshots/SnipS6.PNG}
\label{}
\end{figure} }

\item \label{caB:y1=0} {\bf\boldmath The case where $y_1=0$.} 
Since we also have that $y_2=0$, the ellipse $C_2$ is contained in the $xz$-plane. 
We may set $\bv_2=(x_2,0,z_2)$ in \eqref{eq:AppendixParam2}  that is orthogonal to $\bv_1=(x_1,0,z_1)$. 
After this change, there exists a nonzero $r\in \R$ such that $\bv_2 = r\cdot (-z_1,0,x_1)$.
We compute
\[ 2\coeff[st]-\coeff[s^3t]=-192 (-1 + c) (1 + c) r^2 x_1 z_1^2. \]

Since the above equals zero, $x_1=0$ or $z_1=0$. 
By possibly switching the $x$ and $z$ axes, we may assume that $z_1=0$.
Then, $\bv_1=(x_1,0,0)$ and $\bv_2=(0,0,z_2)$, where $x_1,z_2\neq 0$.
We compute
\[ 2\coeff[st]+\coeff[st^3]=-64 (-1 + c) (1 + c) x_1 (c^2 x_{1}^2 + z_{2}^2 - c^2 z_{2}^2). \]

Since $c>0$ and $x_1,c\neq 1$, we obtain that $c^2 x_{1}^2 + z_{2}^2 - c^2 z_{2}^2=0$. 
This is the same situation as in the middle of Case \ref{caB:x2=y1=0}. 
By repeating the same analysis as in \ref{caB:x2=y1=0}, we obtain matching ellipses.

\ignore{
\begin{figure}[h]
\centering
\includegraphics[width=\textwidth]{MathematicaScreenshots/SnipS8.PNG}
\label{}
\end{figure} }

\item \label{caB:x1x2+z1z2=0} {\bf\boldmath The case where $x_1x_2=-z_1z_2$ and $x_2,y_1\neq0$.} 
We rewrite $x_1x_2=-z_1z_2$ as $x_1=-z_1z_2/x_2$.
We compute $\coeff[1]=4 c x_{2}^7 y_1 (-x_2 + c y_1) (x_2 + c y_1)$. 
Since this equals zero, we get that $x_2=\pm cy_1$.
By potentially reflecting $\R^3$ about the $yz$-plane, we may assume that $x_2= cy_1$.

We swap the positions of $\frac{t^2-1}{t^2+1}$ and $\frac{2t}{t^2+1}$ in \eqref{eq:AppendixParam2}.
This only changes the single point that is not parameterized by \eqref{eq:AppendixParam2}.
Equivalently, we may switch the roles of $\bv_1$ and $\bv_2$ in \eqref{eq:AppendixParam2}.
After this switch, we compute 
$\coeff[1]+\coeff[t^{10}]=2 c^6 y_{1}^6 z_{1}^2 z_{2}^2$.
Since this expression equals zero and $c,y_1\neq 0$, we have that $z_1= 0$ or $z_2= 0$.
In either case, the assumption $x_1x_2=-z_1z_2$ turns to $x_1=0$.
Case \ref{caB:x1=z1=0x2=cy1} assumes that $z_1=0$.
Case \ref{caB:x1=z2=0x2=cy1} assumes that $z_2=0$.

\ignore{
\begin{figure}[h]
\centering
\includegraphics[width=\textwidth]{MathematicaScreenshots/SnipS9.PNG}
\label{}
\end{figure} }

\ignore{
\begin{figure}[h]
\centering
\includegraphics[width=\textwidth]{MathematicaScreenshots/SnipS10.PNG}
\label{}
\end{figure} }

\item \label{caB:x1=z1=0x2=cy1} {\bf\boldmath The case where $x_1=z_1=0$, $x_2=cy_1$, and $x_2,y_1\neq0$.} 
Instead of the switched $\bv_1$ and $\bv_2$ from case \ref{caB:x1x2+z1z2=0}, we return to original vectors. We compute
\[ 2\coeff[s]-\coeff[t]=-24 (-1 + c) c^3 (1 + c) y_{1}^3.\] 
This is nonzero, since $y_1\neq 0$, $c>0$, and $c\neq 1$.
This contradiction implies that the current case cannot occur.

\ignore{
\begin{figure}[h]
\centering
\includegraphics[width=\textwidth]{MathematicaScreenshots/SnipS11.PNG}
\label{}
\end{figure} }

\item \label{caB:x1=z2=0x2=cy1} {\bf\boldmath The case where $x_1=z_2=0$, $x_2=cy_1$, and $x_2,y_1\neq0$.} 
As in Case \ref{caB:x1x2+z1z2=0}, we switch the roles of $\bv_1$ and $\bv_2$.
We then compute  
\[ 4\coeff[st^3]+\coeff[st^5]=48 (-1 + c) c^3 (1 + c) y_1^3. \] 
This is nonzero, since $y_1\neq 0$, $c>0$, and $c\neq 1$.
This contradiction implies that the current case cannot occur.

\ignore{
\begin{figure}[h]
\centering
\includegraphics[width=\textwidth]{MathematicaScreenshots/SnipS12.PNG}
\label{}
\end{figure}
}    
\end{enumerate}

In all cases, we obtained either matching ellipses or a contradiction. 
Thus, two ellipses span many distances unless they are matching or aligned circles.

\section{Mathematica computations: the case of two hyperbolas} \label{app:Wolfram2}

In this appendix, we complete the proof of Lemma \ref{le:EightConicCases} for two hyperbolas, by using Wolfram Mathematica. 
In particular, we assume that all coefficients of $\rho_N(s,t)$ are zero, and show that this implies that the hyperbolas span many distances. 
As explained in Section \ref{sec:FewDistancesBetweenConics}, there exist $c>0$, $p,q,r\in\R$, and vectors $\bv_1 = (x_1,y_1,z_1),\bv_2 = (x_2,y_2,z_2)$, such that  
\begin{align*}
\gamma_1(s) &= \left(\frac{s^2+1}{2s},c\cdot\frac{1-s^2}{2s},0\right),\\
\gamma_2(t) &= (p,q,r)+\frac{t^2+1}{2t}\cdot \mathbf{v}_1+\frac{t^2-1}{2t}\cdot \mathbf{v}_2. 
\end{align*}

We now perform a case analysis, as in Appendix \ref{app:Wolfram}. 
Figure \ref{fi:generalCaseC} includes the code for the beginning of Case \ref{ca:general2}.

 \begin{figure}[h]
    \centering
    \includegraphics[width=0.97\textwidth]{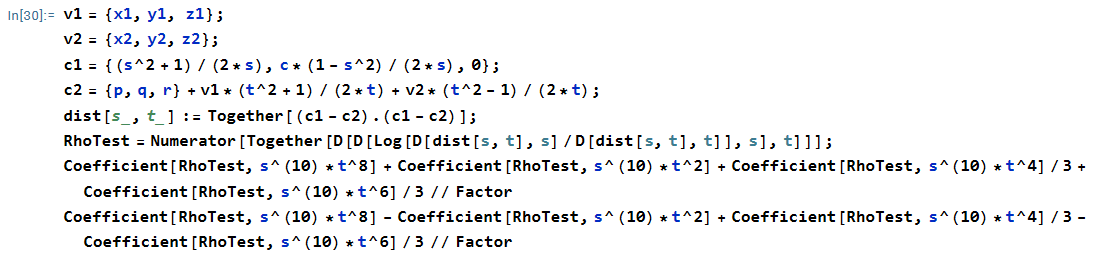}
   \caption{The Mathematica code of Case \ref{ca:general2}. }
\label{fi:generalCaseC}

\end{figure}

\begin{enumerate}[label=(C\arabic*)] 
\item \label{ca:general2}{\bf The general case.} 
For our first step, we compute
\begin{align*}
\coeff[s^{10}t^8]+&\coeff[s^{10}t^2]+\coeff[s^{10}t^4]/3+\coeff[s^{10}t^6]/3= \\
&4 (1 + c^2) (c \cdot y_2-x_2) \left((c\cdot x_1+y_1)^2 + (c\cdot x_2+y_2)^2 + (z_1^2+z_2^2)(1+c^2)\right), \\[2mm]
\coeff[s^{10}t^8]-&\coeff[s^{10}t^2]+\coeff[s^{10}t^4]/3-\coeff[s^{10}t^6]/3= \\
&4 (1 + c^2) (c \cdot y_1-x_1) \left((c\cdot x_1+y_1)^2 + (c\cdot x_2+y_2)^2 + (z_1^2+z_2^2)(1+c^2)\right).
\end{align*}

Since the above expressions equal zero, either $(cx_1+y_1)^2 + (cx_2+y_2)^2 +(z_1^2+z_2^2)(1+c^2)=0$ or $x_1=c\cdot y_1$ and $x_2=c\cdot y_2$.
We first assume that $(cx_1+y_1)^2 + (cx_2+y_2)^2 +(z_1^2+z_2^2)(1+c^2)=0$, which implies that $z_1=z_2=0$.
This in turn implies that the hyperbola $C_2$ is contained in a plane parallel to the $xy$-plane.
By Lemma \ref{le:ParallelPlanes}, the two hyperbolas span many distances.

It remains to consider the case where $x_1=c\cdot y_1$ and $x_2=c\cdot y_2$.
In this case, we compute 
\[ \coeff[t^5] = 16 c (1 + c^2)^2 r (y_2 z_1 - y_1 z_2).\]

Since the above coefficient equals zero, either $r=0$ or $y_2z_1=y_1z_2$.
Case \ref{caB:y1=cx1,y_2z_1=y_1z_2} assumes that $y_2z_1=y_1z_2$.
Case \ref{caB:y1=cx1,r=0} assumes that $y_2z_1\neq y_1z_2$ and $r=0$. 

\item \label{caB:y1=cx1,y_2z_1=y_1z_2} {\bf\boldmath The case where $x_1=cy_1$, $x_2=cy_2$, and $y_2z_1=y_1z_2$.} 
We first assume that $y_1=0$. 
The assumption $x_1=cy_1$ implies that $x_1=0$.
Since $\vb_1$ is non-zero, we get that $z_1\neq 0$. 
Since $\vb_1$ and $\vb_2$ are independent, we have that $y_2\neq 0$ (note that $y_2=0$ also implies that $x_2=0$). 
This contradicts $y_2z_1=y_1z_2$, so $y_1\neq 0$.

We set $z_2 = y_2z_1/y_1$ and compute
\[ \coeff[t^2] - \coeff[t^4]  = -4 c (1 + c^2) y_1^3 \left( y_1^2\cdot(1-c^2)^2 +z_1^2(1+c^2)\right).\]
Since this expression equals zero, we have that $y_1^2\cdot(1-c^2)^2 +z_1^2(1+c^2)=0$.
However, $y_1^2\cdot(1-c^2)^2\ge 0$ and $z_1^2(1+c^2)>0$. 
This contradiction implies that the current
case cannot occur.

\item \label{caB:y1=cx1,r=0} {\bf\boldmath The case where $x_1=cy_1$, $x_2=cy_2$, $y_2z_1\neq y_1z_2$, and $r=0$.} 
We compute
\[ \coeff[st^5] = 128 c^2 (1 + c^2) (y_1 z_2-y_2 z_1) (y_1 z_1 - y_2 z_2).\]
Since this coefficient equals zero, we have that $y_1z_1=y_2z_2$.
Case \ref{caB:y1=cx1,y1=r=0} assumes that $y_1=0$.
Case \ref{caB:y1=cx1,r=0,y1neq0} assumes that $y_1\neq 0$.

\item \label{caB:y1=cx1,y1=r=0} {\bf\boldmath The case where $x_1=cy_1$, $x_2=cy_2$, $y_2z_1\neq y_1z_2$, $y_1z_1=y_2z_2$, and $y_1=r=0$.} 
By the assumption $x_1=cy_1$, we have that $x_1=0$.
Since $\vb_1$ is non-zero, we have that $z_1\neq 0$.
We also have that $y_2\neq 0$, since otherwise the above argument would imply that 
$\vb_1$ and $\vb_2$ are parallel. 
Then, $y_1z_1=y_2z_2$ leads to $z_2 = 0$.

We compute
\[ \coeff[s^6] = -6 c (1 + c^2) y_2 (y_2^2(1+ c^2) + z_1^2)^2.\]
Since $y_2\neq 0$ and $c\neq 0$, the above coefficient is nonzero.
This contradiction implies that the current
case cannot occur.

\item \label{caB:y1=cx1,r=0,y1neq0} {\bf\boldmath The case where $x_1=cy_1$, $x_2=cy_2$, $y_2z_1\neq y_1z_2$, $r=0$, $y_1z_1=y_2z_2$, and $y_1\neq 0$.} 
In this case, we may write $y_1z_1=y_2z_2$ as $z_1 = y_2z_2/y_1$. 
Assume that $z_2 =0$. 
Then $z_1 = y_2z_2/y_1=0$, which implies that the hyperbola $C_2$ is contained in the $xy$-plane.
In this case, Theorem \ref{th:PachDeZeeuw} implies that the two hyperbolas span many distances.

It remains to consider the case where $z_2 \neq 0$.
We compute
\[ \coeff[s^2] = 2 c (y_1 - y_2)^3 \left(y_1^2(1 + c^2) + z_2^2\right) \left(y_1(1-c^2)^2 + z_2^2(1 + c^2)\right).\]
Since the above coefficient equals zero and $y_2,z_2\neq 0$, we get that $y_1=y_2$.
This in turn implies that $y_2z_1 = y_1z_1 = y_2z_2 = y_1z_2$.
This contradiction to the assumption $y_2z_1\neq y_1z_2$ implies that the current
case cannot occur.
\end{enumerate}

The above analysis covers all possible cases, so two hyperbolas always span many distances.

\end{document}